\documentclass[11pt,a4paper,oneside,final,reqno]{amsart}

\usepackage[backend=bibtex, style=alphabetic, sorting=nyt, maxnames=50]{biblatex}
\renewbibmacro{in:}{}

\DeclareFontFamily{OT1}{pzc}{}
\DeclareFontShape{OT1}{pzc}{m}{it}{<-> s * [1.10] pzcmi7t}{}
\DeclareMathAlphabet{\mathpzc}{OT1}{pzc}{m}{it}

\allowdisplaybreaks

\usepackage{comment}
\usepackage[T1]{fontenc}
\usepackage[utf8]{inputenc}
\usepackage[dvipsnames]{xcolor}
\usepackage[english]{babel}
\usepackage{mathtools, amsthm, amssymb, amscd, amsmath} 
\usepackage{newtxtext, newtxmath} 
\usepackage[DIV=11]{typearea} 
\usepackage{hyperref} 
\usepackage{todonotes}
\usepackage{enumitem} 
\usepackage{xypic} 
\usepackage{tikz}
\usepackage[hmargin=2.5cm,vmargin=2.5cm]{geometry}
\usetikzlibrary{cd, decorations.markings, arrows} 
\usepackage{thmtools} 
\usepackage[capitalise]{cleveref}
\usepackage{tensor}

\DeclareMathAlphabet{\mathcal}{OMS}{cmsy}{m}{n} 

\definecolor{DarkPurple}{rgb}{0.40,0.0,0.20}
\hypersetup{
	colorlinks =true,
	linkcolor= DarkPurple, 
	citecolor = NavyBlue 
} 

\usetikzlibrary{cd, decorations.markings, arrows, matrix, calc}

\newtheorem{lemma}{Lemma}[section]
\newtheorem{corollary}[lemma]{Corollary}
\newtheorem{theorem}[lemma]{Theorem}
\newtheorem*{theorem*}{Theorem}
\newtheorem{proposition}[lemma]{Proposition}
\newtheorem{introtheorem}{Theorem}
\newtheorem{assumption}[lemma]{Assumption}

\theoremstyle{definition}
\newtheorem{definition}[lemma]{Definition}
\newtheorem{example}[lemma]{Example}
\newtheorem{remark}[lemma]{Remark}
\newtheorem{question}[lemma]{Question}

\newcommand{\G}{\mathcal{G}}
\newcommand{\Go}{\mathcal{G}^{(0)}}
\newcommand{\E}{\mathcal{E}}
\newcommand{\cK}{\mathcal{K}}
\newcommand{\cX}{\mathcal{X}}
\newcommand{\cY}{\mathcal{Y}}
\newcommand{\cE}{\mathcal{E}}
\newcommand{\curlyL}{\mathcal{L}}
\newcommand{\curlyH}{\mathcal{H}}

\newcommand{\A}{\mathcal{A}}

\newcommand{\supp}{\mathrm{supp}}
\newcommand{\FS}{\mathrm{FS}}

\newcommand{\mk}{\mathrm{mk}}

\newcommand{\Ima}{\mathrm{Im}}
\newcommand{\Lip}{\mathrm{Lip}}

\newcommand{\Cred}{C_{r}^{*}}
\def\acts{\curvearrowright}

\newcommand{\C}{\mathbb{C}}
\newcommand{\R}{\mathbb{R}}

\newcommand{\N}{\mathbb{N}}

\title[]{Quantum metrics from length functions on \'etale groupoids}

\author[1]{Are Austad}

\address{Department of Mathematics, University of Oslo, P.O. Box 1053 Blindern, N-0316 Oslo, Norway}
\email{areaus@math.uio.no}


\numberwithin{equation}{section} 

\begin{document}

	\maketitle

	\begin{abstract}
		We show how to construct a compact quantum metric space from a proper continuous  length function on an \'etale groupoid with compact unit space, where the unit space additionally has the structure of a compact metric space. 
		Using compactly supported Fourier multipliers on the reduced groupoid $C^*$-algebra we provide a sufficient condition for verifying when we obtain a compact quantum metric space in this manner. The condition is sometimes also necessary, and is new even in the case of length functions on discrete groups. 
		Lastly, we show that any AF groupoid with compact unit space can be equipped with a length function from which we obtain a compact quantum metric space, thereby providing a groupoid approach to understanding the quantum metric geometry of unital AF algebras. 
	\end{abstract}

	\section{Introduction}
	The theory of compact quantum metric spaces as pioneered by Rieffel \cite{RieffelMetricsOnStateSpaces1999, RieffelCQMS04, Rieffel2004qGH, Rieffel2004MatrixAlgs, RieffelMartricialBridges2016} extends the classical theory of compact metric spaces to the quantum (or non-commutative) setting. 
	Since its inception, the literature on compact quantum metric spaces has grown vast, seeing contributions from many people. 
	Aligning with current trends in noncommutative geometry \cite{vanSuiljekom21, CvSSpecTrunc21, CvSTolerance22, GielenSuijlekom23, RieffelTruncations23, LeimbachSuijlekom24, KaadExternal24}
	we view a compact quantum metric space as a pair $(\cX, L)$ consisting of a (not necessarily closed) operator system $\cX$ and a $*$-invariant seminorm $L \colon \cX \to [0,\infty)$ with $\C \cdot 1_{\cX} \subseteq \ker L$ and for which the associated Monge-Kantorovi\v{c} metric
	\begin{align*}
		\mk_L(\phi, \psi) = \sup \{ \vert \phi(x) - \psi(x) \vert \mid x\in \cX \text{ and } L(x) \leq 1\}, \quad \phi,\psi \in S(\cX)
	\end{align*} 
	metrizes the weak$^*$ topology on $S(\cX)$, the state space of $\cX$. That these generalize compact metric spaces is discussed in \cite{RieffelCQMS04}.

	A rich source of examples of candidates for compact quantum metric spaces is unital spectral triples. As such, the study of compact quantum metric spaces is strongly linked to Connes' noncommutative geometry \cite{Connes1989, ConnesNCGBook}. Indeed, suppose $(\mathcal{A}, \curlyH, D)$ is a unital spectral triple, where $\mathcal{A}$ is a dense unital $*$-subalgebra of a unital $C^*$-algebra $A$, $\mathcal{H}$ is a separable Hilbert space such that $A \subseteq B(\curlyH)$, and $D$ is a (typically unbounded) self-adjoint operator on $\curlyH$. We then use the bounded commutator condition for spectral triples to define a seminorm
	\begin{align*}
		L(a) = \Vert [D, a]\Vert \quad \text{for $a \in \mathcal{A}$},
	\end{align*}
	and ask if $(\mathcal{A}, L)$ is a compact quantum metric space. 
	Of particular interest to us is the fact that proper  
	length functions on a discrete group $G$ give rise to Dirac operators on $\ell^2(G)$. These in turn yield spectral triples in a natural way, and there are deep results regarding when this construction yields compact quantum metric spaces as above \cite{Rieffel02CQMS, ChristensenIvanRD, OzawaRieffel2004, ChristRieffel}. 
	Similar questions have been considered for compact quantum groups using generalizations of length functions \cite{BhowmickVoigtZacharias2015, AustadKyed2026}, or other spectral triples \cite{AguilarKaad-Podles2018, KaadKyedSU2}. 
	Moreover, recent work has also lent insight into how we might construct compact quantum metric spaces from crossed products \cite{BellissardMarcolliReihani2010, HawkinsSkalskiWhiteZacharias2013, KlisseCQMS, AustadKaadKyed2025}.

	Notably absent from the literature are results on quantum metric structures arising from \'etale groupoids. 
	The main focus of the present article is to remedy this. In particular, we show how an \'etale groupoid $\G$ with compact unit space $\Go$ together with 
	\begin{itemize}
		\item a proper continuous length function $\ell$ on $\G$, and
		\item a metric $d$ on the unit space $\Go$ inducing the topology
	\end{itemize}
	give rise to a candidate $(\cX, L)$ for a compact quantum metric space. More precisely, we find an operator system $\cX$ in the reduced groupoid $C^*$-algebra $\Cred(\G)$ along with a seminorm $L \colon \cX \to [0,\infty)$ from which we may attempt to obtain a compact quantum metric space. 
	The construction jointly generalizes the analogous constructions known for length functions on countable discrete groups, and those for compact metric spaces. 

	It is however not obvious how to even specify an operator system $\cX \subseteq \Cred (\G)$ nor a seminorm, which when combined is sufficiently refined enough  to yield a compact quantum metric space.
	The notion of a length function on a discrete group generalizes straightforwardly to \'etale groupoids, see \cref{def:length-function}. Indeed, a proper continuous length function $\ell$ on an \'etale groupoid $\G$ with compact unit space gives rise to an unbounded operator $D_\ell$ on the Hilbert $C^*$-module $\cE$ naturally associated with the left regular representation $\Lambda$ of $C_c(\G)$. We then further mimic the construction from the discrete group case by defining the seminorm
	\begin{align*}
		L_\ell (f) := \Vert [D_\ell, \Lambda(f)] \Vert .
	\end{align*}
	However, 
	we see that $L_\ell$ trivializes for compact metric spaces $(X,d)$, and so the desired seminorm $L$ must also include the metric on $\Go$.

	Using the fact that $(\Go,d)$ is a compact metric space
	we immediately have access to another seminorm which assigns to a function $f\in C_c(\G)$ the Lipschitz constant of $f_{\vert_{\Go}}$. 
	However, this is again not enough; if $\G = \Gamma \ltimes X$ is a transformation groupoid, we should also consider for $f = \sum_{g\in \G} f_g \lambda_g \in C_c(\G)$ the Lipschitz constants of $f_g \in C(X)$ for each $g \in \Gamma$, \cite{BellissardMarcolliReihani2010, HawkinsSkalskiWhiteZacharias2013, KlisseCQMS, AustadKaadKyed2025}.
	We thus specify 
	a sub-operator system of 
	$C_c(\Gamma \ltimes X)$ 
	which consists of $f$ for which each $f_g$ is Lipschitz continuous. 
	
	Inspired by the decomposition $\G = \cup_{g \in \Gamma} \{g\} \times X$ of a transformation groupoid $\Gamma \ltimes X$, we introduce the notion of a metric stratification of a general \'etale groupoid $\G$ equipped with a metric $d$ on its compact unit space $\Go$. 
	More specifically, a metric stratification is a decomposition $\cK = (K_i)_{i\in I}$ of $\G$ into countably many disjoint subsets for some countable index set $I$, 
	subject to 
	some additional assumptions, 
	see \cref{def:metric-stratification}. 
	Each $K_i$ is a precompact metric space with respect to a naturally appearing metric $d^{(i)}$ on $K_i$.
	Thus every $f \in C_c(\G)$ may be uniquely decomposed as $f = \sum_{i\in I} f_{\vert_{K_i}}$,  
	and by
	measuring the Lipschitz constants of each $f_{\vert_{K_i}}$, $i \in I$, we obtain from $\cK$ a seminorm $L^\cK_{\rm Lip}$, see \eqref{eq:def-decomposition-Lipschitz-norm}.

	There is 
	then 
	a  
	sub-operator system of $C_c(\G)$ given by
	\begin{align*}
		\Lip_c^{\cK}(\G) := \{ f \in C_c(\G) \mid L^\cK_{\rm Lip}(f) < \infty \}.
	\end{align*}
	Combining the above seminorms yields a new seminorm, namely
	\begin{align*}
		L (f) := \max \{ L_\ell(f) , L^{\cK}_{\rm Lip}(f)  \}
	\end{align*}
	for 
	$f \in \Lip_c^{\cK}(\G)$, 
	see \cref{def:total-seminorm}, and we may ask
	\begin{center}
		{\it Is 
			$(\Lip_c^{\cK}(\G), L)$ 
			a compact quantum metric space?}
	\end{center}
	More generally we could ask the same question, but allowing for iterated commutators being used in the definition of $L_\ell$, see \eqref{eq:definition-commutator-seminorm}. 
	
	Having now specified a candidate for a compact quantum metric space arising from an \'etale groupoid equipped with a length function and a metric on its compact unit space, we provide a verifiable condition 
	implying the
	above question has a positive answer in \cref{thm:new-characterization}. This will be done through a novel approach using 
	state-of-the-art results about Fourier multipliers on groupoid crossed products from \cite{BussKwasniewskiMcKeeSkalski2024}.
	In particular, we note that for any $\phi \in C_c(\G)$, there is a completely bounded multiplier 
	\begin{align*}
		m_\phi \colon \Cred(\G) \to \Cred(\G) , \quad m_\phi(f)(\gamma) = \phi(\gamma) f(\gamma), \quad \text{ for }\gamma\in \G,
	\end{align*}
	see \cref{prop:Cc-is-in-FS}.
	We will 
	need that 
	the multipliers 
	are compatible with the chosen metric stratification $\cK$ in the following sense. We say $m_\phi$ is $\cK$-continuous with coefficient $D$ if $L^\cK_{\rm Lip}(m_\phi (f)) \leq D \cdot L^\cK_{\rm Lip}(f)$  for all $f \in \Lip_c^\cK(\G)$. 
	The following is the first main result of the article. 
	\begin{introtheorem}[cf. \cref{thm:new-characterization}]\label{thm:first-intro-theorem} 
		Let $\G$ be an \'etale groupoid with compact unit space $\Go$, and let $\ell \colon \G \to [0,\infty)$ be a proper continuous length function. 
		Suppose moreover $d$ is a metric on $\Go$ inducing its topology,
		and fix a metric stratification $\cK$ of $\G$. Denote by $E$ the set
		\begin{align*}
			E := \{ f \in \Lip_c^{\cK}(\G) \mid L(f) \leq 1 \}.
		\end{align*}
		Consider the statements
		\begin{enumerate}
			\item 
			For every $\varepsilon > 0$ there is 
			$\phi \in C_c(\G)$ 
			such that $m_\phi$ is unital and $\cK$-continuous, and such that
			\begin{align*}
				\sup_{f\in E} \Vert f - m_{\phi}(f) \Vert_{\Cred(\G)} < \varepsilon.
			\end{align*}
			\item $( \Lip_c^{\cK}(\G), L)$ 
			is a compact quantum metric space.
		\end{enumerate}
		Then (1) implies (2).

		Moreover, the converse implication holds if $\G$ admits a sequence of functions
		$(\phi_j)_{j \in \N} \subseteq C_c(\G)$ converging uniformly to $1$ on compact subsets, 
		such that $m_{\phi_j}$ is unital and $\cK$-continuous with coefficient 
		$D_j \geq 0$ 
		for all $j$, and
		satisfying $\sup_j \Vert m_{\phi_j} \Vert < \infty$. 
	\end{introtheorem}
	
	\cref{thm:first-intro-theorem} looks slightly different from \cref{thm:new-characterization}, as we will later consider a seminorm using iterated commutators. 
	In fact, we show that if $\G$ has polynomial growth, or more generally rapid decay, with respect to $\ell$ (see \cref{sec:prelims-groupoids}), then 
	under very mild assumptions
	there is  always $n$ large enough so that we obtain a compact quantum metric space by employing a seminorm using $n$ iterated commutators, see \cref{prop:Christensen-Ivan-for-groupoids}. 
	
	Moreover, \cref{thm:first-intro-theorem} is new even in the setting of discrete groups equipped with proper length functions. For a discrete group the condition relating to continuity with respect to a metric stratification disappears, and the existence of a sequence $(\phi_j)_j$ as in \cref{thm:first-intro-theorem} is in fact guaranteed by weak amenability of the group. Thus we are able to provide a sufficient and necessary criterion for a weakly amenable group equipped with a proper length function to yield a compact quantum metric space as explained above. In particular, this provides a new perspective on the main results of \cite{ChristRieffel, OzawaRieffel2004}, see \cref{cor:FD-CQMS-approx-discrete-groups}.
	
	By employing \cref{thm:first-intro-theorem} we are able to provide the first examples of compact quantum metric spaces arising from groupoids (which do not merely reduce to groups, spaces, or transformation groupoids), namely AF groupoids with compact unit spaces, see \cref{thm:second-intro-theorem}. Such groupoids provide models for unital AF algebras. In particular through the work of Aguilar a lot of effort has gone into understanding  quantum metric aspects of unital AF algebras \cite{AguilarLatremoliere2015, AguilarThesis, AguilarFelltops2019, AguilarIndLim2021}. Using the fact that any unital AF algebra admits a groupoid model, we therefore provide a groupoid approach to understanding the quantum metric geometry of unital AF algebras.
	
	Any AF groupoid with compact unit space arises as the groupoid associated with a Bratteli diagram in a manner which we explain in \cref{sec:AF-groupoids}, and this is the point of view we will take. The Bratteli diagram also naturally induces the length function $\ell$ we will employ, see \cref{def:AF-length-function}, and to our knowledge, this the first time this length function appears in the literature. 
	We have the following result. 
	
	\begin{introtheorem}[cf. \cref{thm:CQMS-from-AF-groupoids}]\label{thm:second-intro-theorem}
		Let $\G$ be an AF groupoid arising from a Bratteli diagram $B$ with finitely many sources. 
		Furthermore, let $\ell$ be the length function arising from $B$ given by \eqref{eq:def-AF-length-function}, and let $d$ be a metric on $\Go$ inducing the topology. 
		Equip $\G$ with the metric stratification given by $\cK = (K_i)_{i \in \N \cup \{0\}}$ where $K_i = \ell^{-1}(\{i\})$. Then $(\Lip_c^{\cK}(\G), L)$  
		is a compact quantum metric space. 
	\end{introtheorem}
	
	Any AF groupoid $\G$ with compact unit space comes equipped with a distinguished sequence of clopen compact principal subgroupoids $(\G_n)_n$. The indicator functions $\phi_n := 1_{\G_n} \in C_c(\G)$ are positive definite, and the associated multipliers are unital and completely positive. We will indeed use the multipliers $m_{\phi_n}$ to prove \cref{thm:second-intro-theorem}. 
	Moreover, for each $n \in \N$ we will find an operator system $\cX_n \subseteq C_c(\G_n)$, and will equip it with the restricted seminorm $L_n := L_{\vert_{\cX_n}}$. 
	The pair $(\cX_n, L_n)$ is a compact quantum metric space for every $n \in \N$, and in fact we show that as $n\to \infty$ the sequence of compact quantum metric spaces $(\cX_n, L_n)_{n\in N}$ converges to $(\Lip_c^\cK(\G), L)$  in the quantum Gromov--Hausdorff distance from \cite{Rieffel2004qGH},  
	see \cref{thm:CQMS-from-AF-groupoids}.

	The article is structured as follows. In \cref{sec:prelims} we collect results about groupoids, Fourier multipliers on groupoid $C^*$-algebras, and fundamental notions from the study of compact quantum metric spaces. 
	In \cref{sec:CQMS-from-groupoids} we introduce metric stratifications of groupoids. Then we show how these can be used to construct a compact quantum metric space from a groupoid with a compact metric space as unit space, and equipped with a  proper continuous length function.  In this section we also prove the first main result of the article, namely \cref{thm:new-characterization}. Lastly, in \cref{sec:AF-groupoids} we demonstrate the first examples of compact quantum metric spaces coming from groupoids by showing that they arise from AF groupoids in a natural way.

 \subsubsection*{Acknowledgments} I  wish to thank Jens Kaad,  Sam Kim, and David Kyed for valuable conversations about compact quantum metric spaces, as well as Eduard Ortega and Mathias Palmstrøm for enlightening discussions about length functions on AF groupoids. 
 I am also grateful to Alistair Miller for comments on a previous version of this article. 
 This research was funded by The Research Council of Norway [project 324944].

	\section{Preliminaries}\label{sec:prelims}
	
	\subsection{Groupoids}\label{sec:prelims-groupoids}
	\subsubsection*{\'Etale groupoids and groupoid actions}
	We first recall some basic material on \'etale groupoids. For the most basic constructions, we refer the reader to \cite{SimsSzaboWilliamsBook}. For actions of groupoids on spaces, we refer the reader to \cite{DelarocheExactGroupoidsV3}.
	
	Throughout the article, all groupoids will be  
	locally compact and Hausdorff. For such a groupoid $\G$, we denote its unit space by $\Go$, the composable pairs by  $\G^{(2)}$, and the range and source maps $r,s \colon \G \to \Go$ are given by $r(\gamma) = \gamma \gamma^{-1}$ and $s(\gamma) = \gamma^{-1}\gamma$, respectively, for $\gamma \in \G$. 
	We say that $\G$ is \'etale whenever the range map $r$ (and therefore also the source map $s$) is a local homeomorphism. 
	If $\G$ is \'etale and $\Go$ is a totally disconnected space, we say that $\G$ is \emph{ample}. 
	
	As is standard, we will for $u \in \Go$ denote by $\G_u = \{ \gamma \in \G \mid s(\gamma) = u \}$ and $\G^u = \{ \gamma \in \G \mid r(\gamma) = u \}$. The groupoid $\G$ is said to be \emph{principal} if $\G_u \cap \G^u = \{u\}$ for every $u \in \Go$.  Moreover, for $H \subseteq \G$, we write $H^{-1} = \{ \gamma^{-1} \in \G \mid \gamma \in H \}$.

	We will need to consider actions of \'etale groupoids on locally compact Hausdorff spaces. First, we say that a locally compact Hausdorff space $X$ is \emph{fibered over} $\Go$ if there is a continuous surjective map $p \colon X \to \Go$, and we say $(X, p)$ is a fiber space over $\Go$. Given two spaces fibered over $\Go$, say $(X_i, p_i)$, $i=1,2$, we may form the fiber product
	\begin{align*}
		X_1 \tensor[_{p_{1}}]{*}{_{p_{2}}} X_2 = \{ (x_1, x_2) \in X_1 \times X_2 \mid p_1(x_1) = p_2(x_2)  \} ,
	\end{align*}
	and we equip it with the relative topology from $X_1 \times X_2$. Note that any groupoid $\G$ is a fibered space over its unit space $\Go$ using either the source map or the range map.

	A \emph{left $\G$-space} is then a fiber space $(X, p)$ over $\Go$ together with a continuous map $(\gamma, x) \mapsto \gamma x$ from $\G \tensor[_s]{*}{_p}  X$ into $X$, satisfying
	\begin{enumerate}
		\item $p(\gamma x) = r(\gamma)$ for all $(\gamma, x) \in \G \tensor[_s]{*}{_p} X$, and $p(x)x = x$ for all $x \in X$;
		\item if $(\gamma, x) \in \G \tensor[_s]{*}{_p} X$ and $(\mu, \gamma)\in \G^{(2)}$, then $(\mu\gamma)x = \mu(\gamma x)$.
	\end{enumerate}
 
		Given a left $\G$-space $(X,p)$, we may construct a new locally compact Hausdorff \'etale groupoid $\G \ltimes X$. As a topological space, it is $\G \tensor[_s]{*}{_p} X$. It is equipped with a product
		\begin{align*}
			(\gamma, \mu  x) (\mu, x) = (\gamma \mu, x),
		\end{align*}
		and inverse
		\begin{align*}
			(\gamma,x)^{-1} = (\gamma^{-1}, \gamma  x).
		\end{align*}
		The resulting range and source maps are given by
		\begin{align*}
			r(\gamma, x) = \gamma  x \quad \text{and} \quad s(\gamma,x) = x.
		\end{align*}
		We identify $(\G \ltimes X)^{(0)}$ with $X$ through the map $(p(x), x) \mapsto x$.
		\begin{remark}\label{remark:left-right-groupoid-isomorphism}
			We will  find it most natural to work with $\G \ltimes X$. However, 
			\cite{DelarocheExactGroupoidsV3} prefers to work with another groupoid arising from a left $\G$-space $(X,p)$, namely $X \rtimes \G$. As a topological space, it is $X \tensor[_p]{*}{_r}\G$, equipped with the product
				\begin{align*}
						(x, \gamma)(\gamma^{-1}x, \mu) = (x, \gamma \mu)
					\end{align*}
				and inverse
				\begin{align*}
						(x, \gamma)^{-1} = (\gamma^{-1}x, \gamma^{-1}),
					\end{align*}
			and resulting range and source maps. $(X \rtimes \G)^{(0)}$ is then identified with $X$ through the map $(x, p(x)) \mapsto x$.  
			Note that there is  a straightforward groupoid isomorphism
			\begin{align*}
				X \rtimes \G &\to \G \ltimes X \\
				(x,\gamma) &\mapsto (\gamma, \gamma^{-1}x),
			\end{align*}
		 	making it easy to translate the results of  \cite{DelarocheExactGroupoidsV3} to our setting.
			 
		\end{remark}
	
	We shall be be especially interested in actions of $\G$ on a (fiberwise) compactification of $\G$. However, the goal of the article is to construct compact quantum metric spaces, and we will therefore have to deal with unital $C^*$-algebras. It is well-known that the algebras $\Cred(\G)$ and $C_c(\G)$ defined below are unital if and only if $\G$ is \'etale and $\Go$ is compact. The \emph{fiberwise Stone-\v{C}ech compactification} $\beta_r \G$ used extensively in \cite{DelarocheExactGroupoidsV3} will in this case therefore coincide with the  Stone-\v{C}ech compactification $\beta \G$, where $\beta \G$ is the usual spectrum of the commutative $C^*$-algebra $C_b (\G)$ of continuous bounded functions on $\G$ (equipped with pointwise product and with conjugation as involution), see \cite[pg. 127]{DelarocheExactGroupoidsV3}. 
	By the universal property of Stone-\v{C}ech compactifications the range map $r \colon \G \to \Go$ extends to a map $r_\beta \colon \beta \G \to \Go$. The pair $(\beta \G, r_\beta)$ is therefore a fiber space over $\Go$. There is also a natural left action of $\G$ on $\beta \G$ given by the continuous extension of the left action of $\G$ on itself. We may therefore consider the groupoid 
	$\G \ltimes \beta \G$, inside which $\G \tensor[_{s}]{*}{_r} \G$ sits as a dense open subspace.

	\subsubsection*{Groupoid algebras}
	Let $C_c(\G)$ denote the space of compactly supported continuous functions on $\G$. It becomes a $*$-algebra when equipped with the convolution product
	\begin{align*}
		f * g (\gamma) = \sum_{\mu \in \G_{s(\gamma)}} f(\gamma \mu^{-1}) g(\mu) = \sum_{\mu \in \G^{r(\gamma)}} f(\mu) g(\mu^{-1}\gamma) \quad \text{for $\gamma \in \G$ and } f,g \in C_c(\G) ,
	\end{align*}
	and involution
	\begin{align*}
		f^*(\gamma) = \overline{f(\gamma^{-1})} \quad \text{for $\gamma \in \G$ and $f \in C_c(\G)$.}
	\end{align*} 
	The \emph{$I$-norm} on $C_c(\G)$ is given by 
	\begin{align}\label{eq:def-I-norm}
		\Vert f \Vert_I :&= \max \left\{ \sup_{u \in \Go} \sum_{\gamma \in \G_u} \vert f(\gamma)\vert , \sup_{u \in \Go} \sum_{\gamma \in \G^u } \vert f(\gamma)\vert \right\}\\ 
		&=  \max \left\{ \sup_{u \in \Go} \sum_{\gamma \in \G_u} \vert f(\gamma)\vert , \sup_{u \in \Go}\sum_{\gamma \in \G_u } \vert f^*(\gamma)\vert \right\}. \notag
	\end{align}
	We will be interested in a $C^*$-completion of $C_c(\G)$, namely the reduced $C^*$-algebra $\Cred(\G)$. In order to construct it, we will need to consider the left regular representation, and we will take the Hilbert $C^*$-module point of view. While we will later need some results from \cite{DelarocheExactGroupoidsV3}, our convention will slightly differ. 
	In \cite{DelarocheExactGroupoidsV3}, the regular representation is implemented using right convolution, and the Hilbert $C^*$-module is constructed using the Hilbert spaces arising as range fibers of the groupoid. However, as this article concerns compact quantum metric spaces, it will be easier to relate our setup and results to the existing literature for groups, see for example  \cite{ChristRieffel, OzawaRieffel2004, ChristensenIvanRD}, if we use the left regular representation on a Hilbert $C^*$-module constructed from the Hilbert spaces arising from the source fibers of the groupoid. These ways of constructing the reduced $C^*$-algebra are equivalent to one another (see \cite[Proposition 8]{PatersonFourierAlgebra2004} and \cite[pg. 81-82]{DelarocheExactGroupoidsV3}), and any result we need from \cite{DelarocheExactGroupoidsV3} will translate to our setting by changing from range fibers to source fibers, and from right convolution to left convolution. Our choice of convention therefore explains the departures from the results from \cite{DelarocheExactGroupoidsV3} in the sequel.  
	We refer the reader
	to \cite{LanceCstarModulesBook} for details on Hilbert $C^*$-modules.
	For a (right) Hilbert $C^*$-module $X$ over a $C^*$-algebra $A$, we denote the $A$-adjointable operators on $X$ by $\mathcal{L}_A (X)$. 
	Fix an \'etale groupoid $\G$ with compact unit space $\Go$. 
	We will denote by $\E$ the right Hilbert $C(\Go)$-module obtained by completing $C_c (\G)$ with respect to the norm coming from the $C (\Go)$-valued inner product
	\begin{align*}
		\langle \xi , \eta \rangle_{C (\Go)} (u) = \sum_{\mu \in \G_u} \overline{\xi(\mu)} \eta(\mu), \quad \xi, \eta \in C_c (\G).
	\end{align*}
	The right $C (\Go)$-module structure on $\E$ is given by 
	\begin{align*}
		(\xi \cdot f)(\gamma) = \xi(\gamma) f(s(\gamma)), \quad \xi \in C_c(\G), f\in C(\Go).
	\end{align*}
	Note that $\E$ is the space of continuous sections of the continuous field of Hilbert spaces with fibers given by $\ell^2(\G_u)$, $u \in \Go$. 
	We then have a left action of $C_c (\G)$ on $\E$ through convolution defined in the following manner
	\begin{align}\label{eq:left-regular-representation}
		(\Lambda (f) \xi) (\gamma) = \sum_{\mu \in \G_{s(\gamma)}} f(\gamma\mu^{-1}) \xi(\mu)
	\end{align}
	The map $\Lambda$ extends to a faithful $*$-representation of $C_c(\G)$ as $C(\Go)$-adjointable operators on $\E$, that is, we obtain an injective $*$-homomorphism $C_c (\G)\to \mathcal{L}_{C(\Go)}(\E)$. Its range is denoted by $\Cred (\G)$ and is known as the \emph{reduced $C^*$-algebra of $\G$.} When $\G$ is \'etale and $\Go$ is compact, both $C_c(\G)$ and $\Cred(\G)$ are unital. 
	We note that the $I$-norm dominates the reduced $C^*$-algebra norm, so we have the inequality 
	\begin{align*}
		\Vert \Lambda (f) \Vert_{\mathcal{L}_{C(\Go)}(\E)} \leq \Vert f \Vert_I 
	\end{align*}
	for all $f \in C_c(\G)$. 

	Furthermore, since $\Go$ is compact we have an action $\tilde{\Lambda}$ of 
	$\Cred(\G \ltimes \beta\G)$ 
	on $\E$ through $C(\Go)$-adjointable operators given by
	\begin{align}\label{eq:action-of-Roe-alg-on-E}
		\tilde{\Lambda}(F)(\xi) (\gamma) = \sum_{\mu \in \G_{s(\gamma)}} \rho(F)( \gamma \mu^{-1}, \mu) \xi(\mu)
	\end{align}
	for $F \in \Cred(\G \ltimes \beta \G)$, 
	$\xi \in \E$, and where $\rho(F)$ is the restriction of $F$ to the dense subset. 
	$\G \tensor[_s]{*}{_r} \G \subseteq  \G \ltimes \beta\G$. We will in the sequel tend to suppress $\rho$ from the notation. 
	In other words, we have a $*$-homomorphism $\tilde{\Lambda} \colon \Cred (\G \ltimes \beta \G ) \to \curlyL_{C(\Go)} (\E)$. The action in \eqref{eq:action-of-Roe-alg-on-E} is compatible with the $*$-algebra inclusion \cite[Lemma 7.5]{DelarocheExactGroupoidsV3}
	\begin{align}\label{eq:inclusion-into-Roe}
		\iota \colon C_c (\G) \hookrightarrow C_c ( \G \ltimes \beta\G), \quad \iota(f)(\gamma, \mu) = f(\gamma)
	\end{align}
	in the sense that $\tilde{\Lambda}\circ \iota (f) = \Lambda (f)$ for all $f \in C_c(\G)$. The map $\iota$ extends to an inclusion $\iota \colon \Cred(\G) \to \Cred(\G \ltimes \beta \G)$.

	When $\G$ is \'etale and $\Go$ is compact, there is always a faithful conditional expectation $P \colon \Cred (\G) \to C (\Go)$ given by restriction of functions, that is
	\begin{align}\label{eq:existence-cond-exp}
		P(f) = f_{\vert_{\Go}} \quad f \in \Cred (\G).
	\end{align}
	Through this conditional expectation we have a way of constructing states on $\Cred(\G)$. More specifically, given any probability measure $\mu$ on $\Go$, we may define a state $\overline{\mu} \in S(C(\Go))$ given by integration against the measure $\mu$. Precomposing this state with the conditional expectation gives a state $\phi_\mu = \overline{\mu} \circ P \in S(\Cred(\G))$.

	Lastly in this section, we remind the reader about the notions of polynomial growth and rapid decay for groupoids. To do so we must define length functions.
	\begin{definition}\label{def:length-function}
		Let $\G$ be an \'etale groupoid. By a \emph{continuous length function on $\G$} we will mean a continuous map $\ell \colon \G \to [0,\infty)$ satisfying
		\begin{enumerate}
			\item $\ell(\gamma) = 0$ if and only if $\gamma \in \Go$,
			\item $\ell(\gamma^{-1}) = \ell(\gamma)$ for all $\gamma \in \G$,
			\item $\ell(\gamma \mu) \leq \ell(\gamma) + \ell(\mu)$ for all $(\gamma, \mu) \in \G^{(2)}$. 
		\end{enumerate}
		The length function is said to be \emph{proper} if the set $\ell^{-1}([0,r])$ is compact for every $r\geq 0$. 
		
		For $t \geq 0$ we define the closed $\ell$-ball of radius $t$ by $B_\ell(t) := \{ \gamma \in \G \mid \ell(\gamma) \leq t\}$. 
	\end{definition}
	Let $\G$ be an \'etale groupoid equipped with a proper continuous length function $\ell$.  
	The
	pair $(\G, \ell)$ 
	will be called 
	a \emph{metric groupoid}. We introduce the following norm on $C_c (\G)$:
	\begin{align*}
		\Vert f \Vert_{2,s} :&= \max \{ \sup_{u \in \Go} \big( \sum_{\gamma \in \G_u} \vert f(\gamma) \vert^2 (1 + \ell(\gamma))^{2s} \big))^{1/2}, \sup_{u \in \Go} \big( \sum_{\gamma \in \G^u} \vert f(\gamma) \vert^2 (1 + \ell(\gamma))^{2s} \big)^{1/2}  \} \\
		&= \max \{ \sup_{u \in \Go} \big( \sum_{\gamma \in \G_u} \vert f(\gamma) \vert^2 (1 + \ell(\gamma))^{2s} \big)^{1/2}, \sup_{u \in \Go} \big( \sum_{\gamma \in \G_u} \vert f^*(\gamma) \vert^2 (1 + \ell(\gamma))^{2s} \big)^{1/2}  \}.
	\end{align*}
	We then say that $(\G, \ell)$ has \emph{rapid decay} if there exist constants $C, s \geq 0$ such that
	\begin{align*}
		\Vert \Lambda (f) \Vert \leq C \Vert f \Vert_{2,s}
	\end{align*}
	for all $f \in C_c (\G)$. 
	
	Examples of metric groupoids with rapid decay are those with polynomial growth \cite[Proposition 3.5]{HouSpectral2017}. We say $(\G, \ell)$ has \emph{polynomial growth} if there is a polynomial $p \colon \R \to \R$ such that
	\begin{align}\label{eq:poly-growth-criterion}
		\sup_{u \in \Go} \vert \{ \gamma \in \G_u \mid \ell(\gamma)\leq t    \}  \vert = \sup_{u \in \Go} \vert B_\ell(t) \cap \G_u \vert \leq p(t) 
	\end{align}
	for all $t \geq 0$. 
	Indeed, by redoing the proof of \cite[Proposition 3.5]{HouSpectral2017} using $p = r + \frac{1}{2} + \varepsilon$ for any $\varepsilon >0$, rather than $p = r+3$, we obtain the following result.
	
	\begin{proposition}\label{prop:degree-of-RD-from-PG}
		Let $(\G,\ell)$ be a metric groupoid, and suppose it has  polynomial growth bounded by some polynomial $p$ of degree $\leq d$. 
		Then for any $s > d + \frac{1}{2}$ there exists $C=C(s)$ such that $(\G, \ell)$ has rapid decay with constants $C,s \geq 0$.
	\end{proposition}

	\begin{assumption}\label{assumption:uniform-fiber-cardinality-bound}
		We will assume that any \'etale groupoid and proper continuous length function considered in this article satisfy that for every $t \geq 0$ there is $F_t \geq 1$ such that
		\begin{align*}
			\sup_{u \in \Go} \vert B_\ell(t) \cap \G_u \vert \leq F_t.
		\end{align*}
		This can be seen to hold true if $(\G,\ell)$ has polynomial growth. Moreover, for discrete groups it is implied by properness of the length function $\ell$. Thus it will also be true for transformation groupoids when using the proper continuous length function from \eqref{eq:length-function-transformation-groupoid}.
	\end{assumption}
	
	\subsection{Multipliers of reduced groupoid $C^*$-algebras} \label{sec:prelims-multipliers}
	In our characterization of metric groupoids yielding compact quantum metric spaces (in a manner which will be explained in \cref{sec:CQMS-from-groupoids}) we will use Fourier multipliers on reduced groupoid $C^*$-algebras. Crucial material is developed and covered in the article \cite{BussKwasniewskiMcKeeSkalski2024}. However, there it is stated in the very general setting of multipliers for twisted groupoid actions and the associated reduced section $C^*$-algebras. For the reader's convenience,   
	we rephrase the results we will need in the setting of reduced groupoid $C^*$-algebras, that is, without the added complexity of twisted actions. 
	We first introduce the notion of a left Hilbert $\G$-bundle for an \'etale groupoid $\G$.  
	Fix an \'etale groupoid $\G$, and suppose moreover that $\Go$ is compact. Let $\curlyH$ be a continuous Hilbert bundle over $\Go$, that is, a continuous field of Hilbert spaces over $\Go$. Denote by $\curlyH_u$ the fiber over $u \in \Go$ and by $\Vert \cdot \Vert_{\curlyH_u}$ the associated norm. 
	By $\Gamma (\curlyH)$ we will mean the continuous bounded sections of $\curlyH$. Then $\Gamma(\curlyH)$ is a Banach space with the norm 
	$\Vert \xi \Vert_{\Gamma(\curlyH)} = \sup_{u \in \Go} \Vert \xi(u) \Vert_{\curlyH_u}$ for $\xi \in \Gamma(\curlyH)$. 
	We say $(\curlyH, L)$ is a $\G$-Hilbert bundle if for each $\gamma \in \G$ there is a linear invertible isometry $L_\gamma\colon \curlyH_{s(\gamma )} \to \curlyH_{r(\gamma)}$, such that for all $\xi, \zeta \in \Gamma(\curlyH)$ the map $\gamma \mapsto \langle L_\gamma \xi(s(\gamma)), \zeta(r(\gamma))\rangle$ is continuous, and such that $L \colon \gamma \mapsto L_\gamma$ is a groupoid homomorphism from $\G $ to the isomorphism groupoid of $\coprod_{u \in \Go} \curlyH_u$.

	With  
	$\G$ as above, for each $\G$-Hilbert bundle $(\curlyH, L)$ and sections $\xi, \zeta \in \Gamma(\curlyH)$, we obtain a function
	\begin{align*}
		C_{\curlyH, L, \xi, \zeta} \colon \G \ni \gamma \mapsto \langle L_\gamma \xi(s(\gamma)), \zeta(r(\gamma))\rangle \in \C .
	\end{align*}
	We denote by $\FS(\G)$ the collection of all $C_{\curlyH, L, \xi, \zeta}$ as we vary over all $\G$-Hilbert bundles $(\curlyH, L)$ and sections $\xi, \zeta \in \Gamma(\curlyH)$. $\FS(\G)$ are the \emph{Fourier-Stieltjes coefficients} of $\G$. We may also equip $\FS(\G)$ with a norm
	\begin{align*}
		\Vert C \Vert_{\FS(\G)} := \inf \{ \Vert \xi \Vert_{\Gamma(\curlyH)} \Vert \zeta \Vert_{\Gamma(\curlyH)} \mid C = C_{\curlyH, L, \xi, \zeta} \text{ for some $(\curlyH, L)$ and $\xi,\zeta \in \Gamma(\curlyH)$} \}.
	\end{align*}
	With this norm, $\FS(\G)$ becomes a Banach space \cite[Proposition 7.10]{BussKwasniewskiMcKeeSkalski2024}. 
	The following result is \cite[Theorem 7.13]{BussKwasniewskiMcKeeSkalski2024} specialized to our setting.
	
	\begin{proposition}\label{prop:FS-yields-cb-maps}
		Any $C \in \FS(\G)$ gives rise to a completely bounded map $m_C \colon \Cred (\G) \to \Cred(\G)$ given by $m_C(f) = C \cdot f$, where $(C \cdot f)(\gamma) = C(\gamma)f(\gamma)$ for $f \in \Cred (\G)$. Moreover, the \emph{completely bounded multiplier norm} $\Vert \cdot \Vert_{\rm cb}$ satisfies $\Vert m_C \Vert_{\mathrm{cb}} \leq \Vert C \Vert_{\FS(\G)}$. If $C = C_{\curlyH, L, \xi, \xi}$, then $m_C$ is completely positive and $\Vert m_C \Vert_{\mathrm{cb}} = \Vert C \Vert_{\FS(\G)} =  \Vert \xi \Vert_{\Gamma(\curlyH)}^2$. 
	\end{proposition}

\begin{remark}\label{remark:multiplier-norm}
	Some results in \cref{sec:CQMS-from-groupoids} can be stated slightly more generally if we use the multiplier norm rather than the completely bounded multiplier norm. So if $C \in \FS(\G)$, we will let $\Vert m_C \Vert$ denote the norm of the Fourier multiplier $m_C \colon \Cred(\G) \to \Cred (\G)$ from \cref{prop:FS-yields-cb-maps}. That is, we do not consider the matrix amplifications to define $\Vert m_C \Vert$. By \cref{prop:FS-yields-cb-maps} we have $\Vert m_C \Vert \leq \Vert m_C \Vert_{\rm cb} \leq \Vert C \Vert_{\FS(\G)}$ for all $C \in \FS(\G)$.  
\end{remark}
	
	We shall also be interested in the fact that positive definite functions on $\G$ give rise to completely positive multipliers on the reduced groupoid $C^*$-algebra. 
	A continuous function $\phi \in C(\G)$ is said to be \emph{positive definite} if for all $u \in \Go$ and all $f \in C_c (\G)$, we have
	\begin{align*}
		\sum_{\mu \in \G^u} \sum_{\gamma \in \G^u} \phi(\mu^{-1}\gamma) f(\mu) \overline{f(\gamma)} \geq 0.
	\end{align*}
	We denote by $P(\G)$ the set of all positive definite functions in $C(\G)$. 
	By \cite[pg. 205-206]{BrownOzawaBook}, $f^* * f \in P(\G)$ for all $f \in C_c(\G)$. 
	Moreover, by \cite[Theorem 1]{PatersonFourierAlgebra2004}, $\phi \in P(\G)$ if and only if $\phi = C_{\curlyH, L, \xi , \xi}$ for a $\G$-Hilbert bundle $(\curlyH , L)$ and a section $\xi \in \Gamma(\curlyH)$. 
	We record the following straightforward result based on \cite[Proposition 3.3]{OtyFourier-Stieltjes}  and its proof. Note that what we have defined to be $\FS(\G)$ corresponds to $B_1(\G)$ in \cite{OtyFourier-Stieltjes}, not $B(\G)$. As such the statement we need does not follow immediately from \cite[Proposition 3.3]{OtyFourier-Stieltjes}, but the proof is identical. 
	\begin{proposition}\label{prop:Cc-is-in-FS}
		Let $\G$ be an \'etale groupoid with compact unit space $\Go$. 
		Then $C_c(\G) \subseteq \FS(\G)$. 
	\end{proposition}
	\begin{proof}
		The proof of \cite[Proposition 3.3]{OtyFourier-Stieltjes} adapts straightforwardly to show that $f * g \in \FS(\G)$ for all $f,g \in C_c(\G)$. Then, since $\Go$ is compact, we have that $1_{\Go} \in C_c(\G)$, from which it follows that $f = f * 1_{\Go} \in \FS(\G)$. 
	\end{proof}
	
	Positive definite functions on \'etale groupoids can in fact be characterized through the resulting completely positive maps on the groupoid $C^*$-algebras, a fact we will use in \cref{sec:AF-groupoids}. 
	The following result is \cite[Theorem 8.5]{BussKwasniewskiMcKeeSkalski2024} specialized to our setting.
	
	\begin{proposition}\label{prop:pos-def-iff-cp-map}
		For a continuous function $\phi \in C(\G)$, the following are equivalent
		\begin{enumerate}
			\item $\phi$ is bounded and positive definite.
			\item $\phi$ induces a completely positive map $m_\phi \colon \Cred (\G) \to \Cred(\G)$ given by
			\begin{align*}
				m_\phi (f) (\gamma) = \phi(\gamma) f(\gamma)
			\end{align*}
			for $f \in C_c (\G)$. 
		\end{enumerate}
	\end{proposition}
	
	Lastly, it will be important for us that we can extend multipliers on the reduced $C^*$-algebra $\Cred (\G)$ to certain reduced $C^*$-algebras associated with actions of $\G$. In particular, if $m_\phi$ is a completely bounded  (resp. completely positive) multiplier on $\Cred (\G)$, we will want it to extend to a completely bounded (resp. completely positive) multiplier on  
	$\Cred ( \G \ltimes \beta\G)$.

	By \cite[Proposition 7.12]{BussKwasniewskiMcKeeSkalski2024}, there is an injective contractive map 
	$\FS(\G) \to \FS (\G \ltimes \beta\G)$,
	obtained by identifying 
	$\Cred ( \G \ltimes \beta\G)$ 
	with the groupoid reduced crossed product $C(\beta \G) \rtimes_{\mathrm{red}} \G$. 
	For $\phi \in \FS(\G)$ we denote the resulting multiplier on $\Cred(\G \ltimes \beta \G)$ by $T^\phi$, and it is given by
	\begin{align}\label{eq:Roe-algebra-lifted-multiplier}
		T^\phi(k) ( \gamma, \mu) = \phi(\gamma) k( \gamma, \mu)
	\end{align}
	for $(\gamma, \mu) \in \G \tensor[_s]{*}{_r} \G \subseteq \G \ltimes \beta \G$ and $k \in \Cred(\G \ltimes \beta \G)$. 
	The proof of \cite[Proposition 7.12]{BussKwasniewskiMcKeeSkalski2024} is done by tensoring by the trivial equivariant representation, therefore (in our setting) sending a Fourier-Stieltjes coefficient $C_{\curlyH, L, \xi, \zeta}$ for $\Cred (\G)$ to the Fourier-Stieltjes coefficient $C_{C(\beta \G) \otimes \curlyH , \alpha \otimes L, 1 \otimes \xi , 1 \otimes \zeta }$ for $C(\beta \G) \rtimes_{\mathrm{red}} \G$. Here $\alpha$ is the action of $\G$ on $C(\beta \G)$ induced by $\G \acts \beta\G$. We therefore see by \cref{prop:pos-def-iff-cp-map} that the map $\FS(\G) \to \FS (\G \ltimes \beta \G )$ sends positive definite functions to positive definite functions. 
	We record the following result.
	\begin{proposition}\label{prop:multipliers-lift-to-crossed-product}
		Let $\phi = C_{\curlyH, L, \xi, \zeta} \in \FS(\G)$. Then  
		${T^\phi} \colon \Cred ( \G \ltimes \beta\G) \to \Cred ( \G \ltimes \beta\G)$ 
		is completely bounded. If $\xi = \zeta$, then ${T^\phi}$ is completely positive. 
	\end{proposition}
	
	\begin{remark}\label{remark:existence-and-boundedness-of-Tphi}
	Note in particular that when $\G$ is \'etale with compact unit space $\Go$, we have $C_c (\G) \subseteq \FS(\G)$ and we have an inclusion $\iota \colon C_c (\G) \hookrightarrow C_c(\G \ltimes \beta \G)$ as in \eqref{eq:inclusion-into-Roe}. Thus if $\phi \in C_c (\G)$, the extension $T^\phi$ exists by \cref{prop:Cc-is-in-FS}, and $\Vert T^\phi \Vert_{\rm cb} < \infty$. Ideally we would be able to bound $\Vert T^\phi \Vert_{\rm cb}$ by $\Vert m_\phi \Vert_{\rm cb}$. If $\phi \in C_c(\G)$ is bounded and completely positive, then it follows from \cref{prop:pos-def-iff-cp-map}, \cref{prop:FS-yields-cb-maps} and \cref{prop:multipliers-lift-to-crossed-product} that
	\begin{align*}
		\Vert T^\phi \Vert_{\rm cb} = \Vert \iota(\phi) \Vert_{\FS(\G \ltimes \beta \G)} \leq \Vert \phi \Vert_{\FS(\G)} = \Vert m_\phi \Vert_{\rm cb}.
	\end{align*}
	Furthermore, if $\G = \Gamma$ is a discrete group, then we could appeal to for example \cite[Corollary 4.7]{BedosConti2015} to deduce that $\Vert T^\phi \Vert_{\rm cb} \leq \Vert m_\phi \Vert_{\rm cb}$ for all $\phi \in C_c(\Gamma)$. The analogous result for \'etale groupoids is not known however, that is, for $\phi \in C_c(\G)$ it is not known if $\Vert T^\phi \Vert_{\rm cb} \leq \Vert m_{\phi}\Vert_{\rm cb}$.
\end{remark}

	\subsection{Compact quantum metric spaces}\label{sec:prelims-CQMS}
	In recent years, a number of techniques in the study of compact quantum metric spaces have been developed using the formalism of operator systems. 
	This is the approach to compact quantum metric spaces we will take in this article. 
	Let us first recall that
	an \emph{operator system} $\cX$ is a unital and $*$-invariant subspace of a unital $C^*$-algebra $A$. The operator system $\cX$ is said to be complete if it is closed in the $C^*$-norm of the ambient $C^*$-algebra $A$. We will often identify the scalars $\C$ with the subspace $\C \cdot 1 \subseteq \cX$ spanned by the unit in $\cX$. 
	An element $x \in \cX$ is positive if it is positive in the ambient $C^*$-algebra, and we define the state space $S(\cX)$ to be the positive, unital functionals on $\cX$. Letting $X := \overline{\cX}^{\Vert \cdot \Vert_A}$ be the closure, we note that $S(\cX)$ and $S(X)$ are homeomorphic through the restriction map.

	The following terminology can be found in \cite{RieffelMartricialBridges2016}. 
	\begin{definition}\label{def:slip-norm}
		Given an operator system $\cX$, a \emph{slip-norm} on $\cX$ is a seminorm $L \colon \cX \to [0,\infty)$ satisfying
		\begin{enumerate}
			\item $L(x^*) = L(x)$ for all $x \in \cX$.
			\item $\C \subseteq \ker L := \{ x \in \cX \mid L(x) = 0 \}$. 
		\end{enumerate}
	\end{definition}
	Given a slip-norm $L$ on $\cX$, we can consider the associated \emph{Monge-Kantorovi\v{c} metric} $\mk_L$ on the state space $S(\cX)$, which is given by
	\begin{align}\label{eq:def-monge-kantorovic-metric}
		\mk_L (\phi, \psi) := \sup \{ \vert \phi(x) - \psi(x) \vert \mid L(x) \leq 1  \} \quad \text{for all $\phi, \psi \in S(\cX)$.}
	\end{align}
	Despite calling $\mk_L$ a metric, it is a priori only an extended metric, that is, $\mk_L$ may assign infinite distances between states. 
	\begin{definition}\label{def:CQMS}
		Suppose $\cX$ is an operator system equipped with a slip-norm $L$. If the (extended) metric $\mk_L$ of \eqref{eq:def-monge-kantorovic-metric} metrizes the weak$^*$ topology on $S(\cX)$, we say that $(\cX, L)$ is a \emph{compact quantum metric space}. 
	\end{definition}
	
	\begin{remark}\label{remark:CQMS-implies-kernel-is-C}
		Suppose $L$ is a slip-norm on an operator system $\cX$. If $(\cX, L)$ is a compact quantum metric space, then it must in fact be true that $\ker (L) = \C$. To see this, note that if there were $x \in \ker (L) \setminus \C$, then we could find states $\phi, \psi \in S(\cX)$ with $\phi(x) \neq \psi(x)$ and then
		\begin{align*}
			\sup \{ \vert \phi(y) - \psi(y) \vert \mid L(y) \leq 1  \} = \infty 
		\end{align*}
		since $tx \in \ker(L)$ for all $t \in \R$. But $S(\cX)$ is compact and connected in the weak$^*$ topology, and so it can not be metrized by an extended metric attaining the value $\infty$. 
	\end{remark}

	The statement in \cref{def:CQMS} is difficult to verify in practice. 
	We record the following useful reformulation from \cite{OzawaRieffel2004}, which we will make extensive use of in the sequel. Note that in \cite{OzawaRieffel2004} the result is stated for dense $*$-subalgebras of unital $C^*$-algebras. However, the result and proof is based on \cite[Theorem 1.8]{RieffelMetricActionsCompactGroups1998}, wherein only an operator system structure is required. We therefore state the result for operator systems. Recall that a subset of a metric space is said to be \emph{totally bounded} if for any $\varepsilon >0$, it may be covered by finitely many $\varepsilon$-balls. If the metric space is complete, then total boundedness is equivalent to precompactness.

		\begin{proposition}[Proposition 1.3 in \cite{OzawaRieffel2004}]\label{prop:Ozawa-Rieffel-characterization}
		Let $\cX$ be an operator system in an ambient unital $C^*$-algebra $A$, and let $L \colon \cX \to [0,\infty)$ be a slip-norm.  Let $\sigma$ be any state on $\cX$.  
		Then 
		$(\cX,L)$ is a compact quantum metric space if and only if
		\begin{align*}
				E^\sigma_L = \{ x \in \A \mid L(x) \leq 1 \text{ and } \sigma(x) = 0 \}
			\end{align*}
		is a norm-totally bounded subset of $A$.  
	\end{proposition}
	
	\begin{remark}\label{remark:smaller-seminorm-CQMS-implies-CQMS}
		Suppose $\cX$ is an operator system in an ambient unital $C^*$-algebra $A$, and let $L_1, L_2 \colon \cX \to [0,\infty)$ be slip-norms. Assume that $L_1(x) \leq L_2(x)$ for all $x \in \cX$. As subsets of totally bounded sets are totally bounded, we see by \cref{prop:Ozawa-Rieffel-characterization}  that if $(\cX, L_1)$ is a compact quantum metric space, so is $(\cX, L_2)$.
	\end{remark}

	The following two examples show how we may attempt to construct compact quantum metric spaces from discrete groups equipped with length functions, and from classical compact metric spaces. Indeed, as \'etale groupoids may be viewed as joint generalizations of discrete groups and locally compact Hausdorff spaces, these examples will provide the inspiration for how to construct compact quantum metric spaces 
	in \cref{sec:CQMS-from-groupoids}.

	\begin{example}\label{example:CQMS-from-groups}
		Let $\Gamma$ denote a countable discrete group with unit $e$. Suppose $\ell \colon \Gamma \to [0,\infty)$ is a length function, by which we mean $\ell$ satisfies $\ell(g) = 0$ if and only if $g = e$, $\ell(g^{-1}) = \ell(g)$ for all $g \in \Gamma$, and $\ell(g_1g_2) \leq \ell(g_1) + \ell(g_2)$ for all $g_1,g_2 \in \Gamma$. Moreover, suppose $\ell$ is proper, that is, $\ell^{-1}([0,r])$ is finite for all $r \geq 0$. 
		The proper length function $\ell$ gives rise to a self-adjoint unbounded operator given by the self-adjoint closure of the operator
		\begin{align*}
			D_\ell \colon C_c (\Gamma) \to \ell^2(\Gamma), \quad \delta_g \mapsto \ell(g) \delta_g 
		\end{align*}
		where $(\delta_g)_{g \in \Gamma}$ is the canonical basis for $\ell^2(\Gamma)$. We also denote the closure of the above operator by $D_\ell$ for simplicity. The triple $(C_c (\Gamma), \ell^2 (\Gamma), D_\ell)$ then defines a spectral triple for the $C^*$-algebra $\Cred (\Gamma)$ \cite{Connes1989}. We obtain a derivation
		\begin{align*}
			\delta \colon C_c (\Gamma) \to B(\ell^2(\Gamma)), \quad f \mapsto \overline{[D_\ell, \Lambda(f)]},
		\end{align*}
		where $\Lambda \colon \Cred(\Gamma) \to B(\ell^2(\Gamma))$ is the left regular representation. By $\overline{[D_\ell, \Lambda(f)]}$ we mean the closure of the operator $[D_\ell, \Lambda(f)]$.  In \cite{Rieffel02CQMS} Rieffel uses the derivation $\delta$ to construct a slip-norm on $C_c(\Gamma)$. Note that here $C_c (\Gamma)$ is realized as an operator system inside the $C^*$-algebra $\Cred (\Gamma)$. Moreover, $C_c(\Gamma)$ is dense in $\Cred(\Gamma)$, so $S(\Cred (\Gamma)) \cong S(C_c(\Gamma))$ through the restriction map. Specifically, for $f \in C_c(\Gamma)$, he sets 
		\begin{align}\label{eq:slip-norm-length-discrete-group-case}
			L_\ell (f) := \Vert \delta(f) \Vert_{B(\ell^2(\Gamma))}
		\end{align}
		and asks whether the pair $(C_c (\Gamma), L_\ell)$ gives rise to a compact quantum metric space. 
		This has been shown in the positive for word-hyperbolic groups \cite{OzawaRieffel2004} and for metric groups $(\Gamma, \ell)$ with bounded doubling (in particular finitely generated groups of polynomial growth) \cite{ChristRieffel}, and there are currently no known counterexamples using word-length functions. In \cite{ChristensenIvanRD}, they consider the analogous question but allow for iterated application of the derivation $\delta$. That is, considering the slip-norm $L^n_\ell(f) := \Vert \delta^n (f) \Vert_{B(\ell^2(\Gamma))}$ for $n \in \N$, they ask when the pair $(C_c (\Gamma), L^n_\ell)$ yields a compact quantum metric space. 
		It is shown that whenever the metric group $(\Gamma, \ell)$ has rapid decay, there is $k \in \N$ such that $(C_c (\Gamma), L^n_\ell)$ is a compact quantum metric space for all $n \geq k$. Moreover, $k$ can be determined by the rate of rapid decay for $(\Gamma, \ell)$, that is, by the quantity
		\begin{align*}
			\inf \{ s\geq 0\mid \text{There is $C > 0$ such that $\Vert f \Vert_{\Cred(\Gamma)} \leq C \Vert f \Vert_{2,s}$ for all $f \in C_c(\Gamma)$}\}.
		\end{align*}
	\end{example}
	
	\begin{example}\label{example:CQMS-from-compact-metric-spaces}
		Let $(X,d)$ be a compact metric space, and denote by $\mathrm{Lip}(X)$ the dense $*$-subalgebra of $C(X)$ consisting of Lipschitz continuous functions on $(X,d)$. Let $L \colon \mathrm{Lip}(X) \to \R$ be the assignment of a function to its Lipschitz constant, that is
		\begin{align*}
			L(f) = \sup \{ \frac{\vert f(x) - f(y) \vert}{d(x,y)} \mid x \neq y  \}
		\end{align*}
		for $f \in \mathrm{Lip }(X)$. It is easily verified that $L$ is a slip-norm, and in fact $(\mathrm{Lip}(X), L)$ becomes a compact quantum metric space, because 
		\begin{align*}
			\mk_L (\phi, \psi) = \sup \{ \vert \phi(f) - \psi(f) \vert \mid f\in \mathrm{Lip}(X) \text{ with } L(f) \leq 1  \}
		\end{align*}
		recovers the weak$^*$-topology on $S(C(X))$, see \cite{RieffelMetricsOnStateSpaces1999}. 
	\end{example}
	
	In \cref{sec:AF-groupoids}, we will show that compact quantum metric spaces determined by certain subgroupoids of AF groupoids converge to the compact quantum metric space determined by the AF groupoid in a certain topology. This topology is determined by the quantum Gromov--Hausdorff distance, which assigns a distance between pairs of compact quantum metric spaces. In order to discuss this, we must introduce the notion of admissible seminorms on direct sums of operator spaces, for which we must do a brief detour. Our definition of quantum Gromov--Hausdorff distance will agree with Rieffel's original notion from \cite{Rieffel2004qGH}, but for our purposes we will find it easier to follow the discussion and conventions from \cite{KaadKyedSU2}. We remark that there are other notions of distance which could be relevant. 
	We mention in particular the quantum Gromov--Hausdorff propinquity by Latrémolière, see for example 
	\cite{LatremoliereDual15, LatremoliereqGHP16}. However, we take the operator system approach to compact quantum metric spaces in this article. As such the Leibniz seminorms used for the quantum Gromov--Hausdorff propinquity do not appear naturally. It could be interesting to investigate 
	to which extent
	the constructions of this article 
	might be adapted to Latrémolière's $C^*$-algebraic formulation of quantum metric geometry. 

	Now, 
	suppose $\cX$ is an operator system, and let $L_\cX\colon \cX \to [0,\infty)$ be a slip-norm.   Denote by $\cX_{\mathrm{sa}}$  the selfadjoint elements of $\cX$. Then $L_\cX$ restricts to a seminorm $(L_\cX)_{\mathrm{sa}} \colon \cX \to [0,\infty)$. 
	
	Suppose we are now given two compact quantum metric spaces $(\cX, L_\cX)$ and $(\cY, L_\cY)$. Note that $\cX \oplus \cY$ is an operator system in a natural way. 
	A slip-norm $K \colon \cX \oplus \cY \to [0,\infty)$ is said to be \emph{admissible} when $(\cX \oplus \cY, K)$ is a compact quantum metric space and the quotient seminorms induced by $K_\mathrm{sa}$ via the coordinate projections $\cX \oplus \cY \to \cX$ and $\cX \oplus \cY \to \cY$ are  $(L_\cX)_\mathrm{sa}$ and $(L_\cY)_\mathrm{sa}$, respectively. 
	Whenever $K \colon \cX \oplus \cY \to [0,\infty)$ is admissible, the coordinate projections induce isometric inclusions of compact metric spaces $(S(\cX), \mk_{L_{\cX}}) \to (S(\cX \oplus \cY), \mk_{K})$ and $(S(\cY), \mk_{L_{\cY}}) \to (S(\cX \oplus \cY), \mk_K)$. In particular, each admissible slip-norm $K$ gives rise to a distance between $(S(\cX), \mk_{L_{\cX}})$ and $(S(\cY), \mk_{L_{\cY}})$ through the Hausdorff distance which we denote by
	\begin{align*}
		\mathrm{dist}^{d_K}((S(\cX), \mk_{L_{\cX}}), (S(\cY), \mk_{L_{\cY}})).
	\end{align*}
	In analogy with the classical Gromov--Hausdorff distance for compact metric spaces, we define the \emph{quantum Gromov--Hausdorff distance} between $(\cX, L_\cX)$ and $(\cY, L_\cY)$ as
	\begin{align*}
		\mathrm{dist}_Q &((\cX, L_\cX), (\cY, L_\cY))\\ &:= \inf \{ \mathrm{dist}^{d_K}((S(\cX), \mk_{L_{\cX}}), (S(\cY), \mk_{L_{\cY}})) \mid K \colon \cX \oplus \cY \to [0, \infty) \text{ admissible slip-norm}  \}.
	\end{align*}
	Exact distances are difficult to calculate. We shall however mostly be concerned with convergence of compact quantum metric spaces obtained by restriction of the quantum metric structure from an ambient system. 
	We record the following result, which is a special case of \cite[Corollary 2.2.5]{KaadKyedSU2}. 
	\begin{proposition}\label{prop:KK-subsystem-qGH-convergence}
		Let $(\cX, L_\cX)$ be a compact quantum metric space, and suppose $\cY \subseteq \cX$ is a sub-operator system. Let $L_\cY = (L_\cX)_{ \vert_{\cY}}$. Suppose there is $\varepsilon > 0$ along with a unital positive map $\Phi \colon \cX \to \cY$ such that $L_\cY (\Phi (x)) \leq L_\cX (x)$ and $\Vert x - \Phi(x)\Vert_\cX \leq \varepsilon \cdot L_{\cX} (x)$ for all $x \in \cX$. Then $(\cY, L_\cY)$ is a compact quantum metric space, and
		\begin{align*}
			\mathrm{dist}_Q ((\cX, L_\cX), (\cY, L_\cY)) \leq  \varepsilon.
		\end{align*} 
	\end{proposition}

	\section{Compact quantum metric spaces from groupoids}\label{sec:CQMS-from-groupoids}
	\subsection{Constructing the seminorm and operator system}
	As an \'etale groupoid with compact unit space generalizes both a discrete group and a compact metric space, our candidate $(\cX, L)$ for a compact quantum metric space consisting of an operator system $\cX \subseteq \Cred (\G)$ and slip-norm $L\colon \cX \to [0,\infty)$, should generalize both \cref{example:CQMS-from-groups} and \cref{example:CQMS-from-compact-metric-spaces}. We do this in two steps. Let us first see how the slip-norm coming from the length function in \cref{example:CQMS-from-groups} generalizes to the groupoid setting. We remind the reader that all groupoids considered in this article will be 
	locally compact and Hausdorff.
	
	Now, let us be given an \'etale groupoid $\G$ with compact unit space $\Go$, along with a continuous 
	proper length function $\ell\colon \G \to [0,\infty)$.
	Let $\Lambda \colon \Cred (\G) \to \curlyL_{C(\Go)}(\E)$ be the left regular representation, see \cref{sec:prelims-groupoids}. Realizing $C_c(\G)$ as a subspace of $\E$,
	the length function $\ell$ gives rise to an operator $D_\ell \colon  C_c (\G) \to \E$ given by
	\begin{align}\label{eq:def-action-of-Dell}
		(D_\ell \xi) (\gamma) = \ell(\gamma) \xi(\gamma) 
	\end{align}
	for $\xi \in C_c (\G)$. 
	The calculation
	
	\begin{align*}
		D_\ell(  \xi \cdot f)(\gamma) = \ell(\gamma) (\xi \cdot f)(\gamma) = \ell(\gamma) \xi(\gamma)f(s(\gamma)) =  \ell(\gamma) \xi(\gamma) f(s(\gamma)) =  ((D_\ell \xi)\cdot f)(\gamma)
	\end{align*}
	for $\gamma \in \Gamma$, $f \in C(\Go)$ and $\xi \in \E$, shows that $D_\ell$ is a
	densely defined 
	(typically unbounded) operator on the Hilbert $C(\Go)$-module $E$. 
	
	The analogue of \cref{example:CQMS-from-groups} would be to now consider the commutator $[D_\ell , \Lambda (f)]$ for $f \in C_c (\G)$, and the resulting seminorm $L_\ell (f) = \Vert [D_\ell, \Lambda (f)]\Vert$. We first verify that this seminorm, and the analogous one obtained by taking iterated commutators, is in fact finite for every $f \in C_c (\G)$.
	
	\begin{proposition}\label{prop:iterated-commutator-adjointable}
		Suppose $f \in C_c (\G)$, and let $\delta^n (f)$ for $n \in \N$ denote the $n$ times iterated commutator $[D_\ell, [D_\ell, \ldots [D_\ell, \Lambda(f)]]]$. Then for any $\xi \in \E$ and $\gamma \in \Gamma$, we have
		\begin{align*}
			\delta^n(f)(\xi)(\gamma) = \sum_{\mu \in \G_{s(\gamma)}} (\ell(\gamma) - \ell(\mu))^n f(\gamma \mu^{-1}) \xi(\mu).
		\end{align*}
		Consequently we may view $\delta^n(f)$ as an element of the groupoid Roe algebra 
		$\Cred( \G \ltimes \beta\G)$,
		and so we may also view $\delta^n(f) \in \mathcal{L}_{C(\Go)}(\E)$. 
	\end{proposition}
	\begin{proof}
		A straightforward calculation using \eqref{eq:left-regular-representation} and \eqref{eq:def-action-of-Dell} will show that for $n=1$ we have
		\begin{align*}
			\delta(f)(\xi)(\gamma) = \sum_{\mu \in \G_{s(\gamma)}} (\ell(\gamma) - \ell(\mu)) f(\gamma \mu^{-1}) \xi(\mu)
		\end{align*}
		for all $f \in C_c(\G)$, $\xi \in \E$ and $\gamma \in \G$. Suppose now that the identity from the statement of the proposition holds for $n-1$. We then calculate
		\begin{align*}
			\delta^n(f)(\xi)(\gamma) &= [D_\ell, \delta^{n-1}(f)](\xi)(\gamma) \\
			&=D_\ell \delta^{n-1}(f)(\xi)(\gamma) - \delta^{n-1}(f) (D_\ell \xi)(\gamma) \\
			&= \ell(\gamma) \sum_{\mu \in \G_{s(\gamma)}} (\ell(\gamma) - \ell(\mu))^{n-1}f(\gamma \mu^{-1})\xi(\mu) - \sum_{\mu \in \G_{s(\gamma)}} (\ell(\gamma) - \ell(\mu))^{n-1}f(\gamma \mu^{-1})\ell(\mu)\xi(\mu) \\
			&= \sum_{\mu \in \G_{s(\gamma)}} (\ell(\gamma) - \ell(\mu))^{n}f(\gamma \mu^{-1})\xi(\mu),
		\end{align*}
		which shows the first part of the proposition. We proceed to show that  $\delta^n(f)$ has a unique extension to an element of 
		$C_c (\G \ltimes \beta\G)$. 
		Let $F\colon \G \tensor[_s]{*}{_r} \G \to \C$ be given by 
		\begin{align*}
			F(\gamma, \mu) = (\ell(\gamma \mu) - \ell(\mu))^n f(\gamma).
		\end{align*}
		If 
		$F \in C_c ( \G \ltimes \beta\G)$,
		we see by \eqref{eq:action-of-Roe-alg-on-E} 
		that $\delta^n (f) \xi = \tilde{\Lambda}(F) \xi$ for all $\xi \in \E$. 
		To see that (the unique extension of) $F$ is in 
		$C_c (\G \ltimes \beta \G)$,
		consider first that through the isomorphism of \cref{remark:left-right-groupoid-isomorphism}, $F$ corresponds to the function $\hat{F}$ on $\beta \G \rtimes \G$ given by 
		\begin{align*}
			\hat{F}(\mu, \gamma) = F(\gamma, \gamma^{-1}\mu) = (\ell(\mu) - \ell(\gamma^{-1}\mu))^n f(\gamma) 
		\end{align*}
		Then, consider the function $\theta(\hat{F})$ on $\G \tensor[_r]{*}{_r} \G$ given by $\theta(\hat{F})(\mu, \gamma) = \hat{F}(\mu^{-1}, \mu^{-1}\gamma)$, that is
		\begin{align*}
			\theta(\hat{F})(\mu, \gamma) &= (\ell(\mu^{-1}) - \ell((\mu^{-1}\gamma)^{-1}\mu^{-1}))^n f(\mu^{-1}\gamma)\\
			&= (\ell(\mu) - \ell(\gamma))^n f(\mu^{-1}\gamma).
		\end{align*}
		Note that $\theta(\hat{F})$ is continuous on $\G\tensor[_r]{*}{_r} \G$ since both $f$ and $\ell$ are continuous. 
		Since the support $\supp f$ is compact, we see that $\theta(\hat{F})$ has support on a \emph{tube}, that is a subset of $\G \tensor[_r]{*}{_r} \G$ for which the map $(\gamma, \mu) \mapsto \gamma^{-1}\mu$ is precompact in $\G$.  It follows by \cite[Lemma 6.17]{DelarocheExactGroupoidsV3} that (the unique extension of) $\hat{F}$ is in $C_c (\beta \G \rtimes \G)$, and therefore $F \in C_c (\G \ltimes \beta \G)$. Since $\Cred( \G \ltimes \beta\G)$ acts on $\E$ through adjointable operators, this finishes the proof.   
	\end{proof}
	
	By \cref{prop:iterated-commutator-adjointable} $\delta^n (f) \in \mathcal{L}_{C(\Go)}(\E)$ for any $f \in C_c(\G)$ and any $n \in \N$. A straightforward but tedious calculation will show that
	\begin{align*}
		\delta^n (f)^* = (-1)^k \delta^n(f^*)
	\end{align*}
	so that in particular
	\begin{align*}
		\Vert \delta^n (f^*) \Vert = \Vert \delta^n(f) \Vert 
	\end{align*}
	for all $f \in C_c(\G)$. 
	In light of 
	this
	we may, for each $n \in \N$, define a seminorm 
	\begin{align}\label{eq:definition-commutator-seminorm}
		L_\ell^{n}(f) := \Vert \delta^n (f)\Vert_{\curlyL_{C(\Go)}(\E)} \quad \text{for }f \in C_c(\G).
	\end{align}
	We then see that $L_\ell^n$ is $*$-invariant, and 
 $\C  \subseteq \ker L_\ell^n$. 
	The seminorm $L_\ell^n$ is therefore a slip-norm in the sense of \cref{def:slip-norm}. 
	Moreover,
	in the case of discrete groups we recover the seminorm from \cref{example:CQMS-from-groups}. An immediate question we may ask
	is then if $(C_c (\G), L_\ell^n)$ can be shown to be a compact quantum metric space (for $n$ large enough). However, for any \'etale groupoid with compact unit space larger than a just a point this is not going to work for a very simple reason: the unit space is too large. More precisely, let $f \in C_c (\G)$ be any function such that $f = P(f)$, where $P$ is the conditional expectation from \eqref{eq:existence-cond-exp}. That is, the  support of $f$ is contained in $\Go$. It is then not difficult to see that $L_\ell^n (f) = 0$, but if $\Go$ is not just a point, there definitely exist such functions which are not just scalar multiples of the unit. Therefore $\ker L_\ell^n \supsetneq \C$ in these cases, and by \cref{remark:CQMS-implies-kernel-is-C} $(C_c (\G), L_\ell^n (f))$ can not be a compact quantum metric space for any $n \in \N$. This is also easily seen by realizing that the seminorm we construct should be able to cover recover \cref{example:CQMS-from-compact-metric-spaces}, and for any compact metric space $(X,d)$ we would have $L_\ell^n (f) = 0$ for all $f \in C_c (\G) = C(X)$. To remedy this flaw, we will take inspiration from \cite{AustadKaadKyed2025} to incorporate the unit space $\Go$ into the seminorm. To do this we introduce the following technical notion, which will play a key role in the constructions in the remainder of the article.
	
	\begin{definition}\label{def:metric-stratification}
		Suppose $\G$ is an \'etale groupoid with compact unit space $\Go$, and suppose 
		$d$ is a metric on $\Go$ inducing the topology, so $(\Go, d)$ is a compact metric space. 
		A \emph{metric stratification of $\G$ with respect to $d$} is a collection $\cK = \{K_i\}_{i\in I}$, where $I$ is a countable index set, for which the following holds:
		\begin{enumerate}
		 	\item $\G = \bigcup_{i \in I} K_{i}$.
		 	\item We have $K_{i_{1}} \cap K_{i_{2}} = \emptyset$ for $i_1\neq i_2$.
		 	\item Each $K_{i}$ is precompact 
		 	and open.
		 	\item There is a distinguished element $K_{e} \in \cK$ for which $K_{e}  = \Go$.
		 	\item For every $K_{i} \in \cK$, we also have $K_{i}^{-1}\in \cK$. 
		 	\item For each $i\in I$, the map 
		 	\begin{align}\label{eq:di-is-a-metric}
		 		d^{(i)}(\gamma, \mu) = \max \{ d(s(\gamma), s(\mu)) , d(r(\gamma), r(\mu))  \}, \quad \gamma, \mu \in K_{i},
		 	\end{align}
		 	induces on $K_i$ the structure of a totally bounded metric space. 
		 	
		\end{enumerate}
		If the metric $d$ is implied, we simply say that $\cK$ is a \emph{metric stratification of $\G$}.
	\end{definition}

\begin{remark}\label{remark:comments-on-metric-stratification-def}
	\begin{enumerate}
		\item Note that \cref{def:metric-stratification} does not rule out the possibility that $K_i = \emptyset$ for some $i \in I$.
		\item To verify condition (6) of \cref{def:metric-stratification} it suffices to show that for each $i \in I$, we have $d^{(i)}(\gamma, \mu) = 0 $ if and only if $\gamma = \mu$. Indeed, if this holds, then \eqref{eq:di-is-a-metric} tells us that $(K_i , d^{(i)})$ can be realized as a subspace of $(\Go \times \Go, \tilde{d})$ through the map $s \times r$, where
		\begin{align*}
			\tilde{d}((x_1, y_1) , (x_2,y_2)) = \max \{ d(x_1,x_2) , d(y_1,y_2) \} \quad \text{for $(x_1,y_1), (x_2,y_2) \in \Go \times \Go$}.
		\end{align*}
		Since $(\Go \times \Go, \tilde{d})$ is a compact metric space, we deduce that $d^{(i)}$ induces on $K_i$ the structure of a totally bounded metric space. 
	\end{enumerate}
	 
\end{remark}
	
	The conditions in \cref{def:metric-stratification} deserve an explanation. We illustrate first that they appear naturally from considering a transformation groupoid 
	$\G = \Gamma \ltimes X$ 
	where $\Gamma$ is a countable group, and $(X,d)$ is a compact metric space. Let $\ell_\Gamma$ be a proper length function on $\Gamma$, and induce on  
	$\Gamma \ltimes X$ 
	the proper continuous length function $\ell$ given by 
	\begin{align}\label{eq:length-function-transformation-groupoid}
			\ell(g,x) = \ell_\Gamma (g) \quad \text{ for $(g,x) \in \Gamma \ltimes X$.}
	\end{align}
	We would then immediately get a very natural looking decomposition satisfying all the above conditions by setting $\cK = (K_g)_{g\in \Gamma}$ where 
	$K_g =  \{g\} \times X$. 
	Notably, in this case we would have that each $K_g$ is compact and open, but we have weakened the condition to only require precompact in order to accommodate more general groupoids. 
	However, suppose that $\ell_\Gamma$ is integer-valued. Then another decomposition which would satisfy conditions (1)-(5) is given by setting $\cK = (K_n)_{n \in \N\cup \{0\}}$ where $K_n = \ell^{-1}(\{n\})$. However, it is entirely possible there could be $g_1, g_2 \in \Gamma$ with $\ell_\Gamma (g_1) = \ell_\Gamma(g_2)$, and $x \in X$ for which $g_1 x = g_2 x$. Then
	\begin{align*}
		r(g_1,x) = r(g_2,x) \quad \text{and} \quad s(g_1,x) = s(g_2,x)
	\end{align*}
	and the map in condition (6) would not be a metric. The absence of condition (6) would make analysis difficult in the sequel. 
	Indeed, we may think of condition (6) as a substitute for requiring that the decomposition consists of bisections, that is, subsets of $\G$ for which the range and source maps are local homeomorphisms. Were we to require $\cK$ to consist of bisections however, we would find few examples where condition (2) is simultaneously satisfied, and this condition likewise plays an important role in the sequel.  
	The main results of this paper hinges crucially on the important \cref{lemma:compact-subsets-yield-CQMS}, and several of the conditions in \cref{def:metric-stratification} can be understood from how the proof of this lemma will be done. Let us therefore sketch the overall procedure. 
	In the proof of this \cref{lemma:compact-subsets-yield-CQMS}, we will want to take $f \in \Lip_c^{\cK}(\G)$ and uniquely decompose it into a finite sum $\sum_i f_{\vert_{K_i}}$, for which conditions (1) and (2) play a crucial role. 
	We then want to treat each $K_i$ as a precompact metric space with the metric coming from \eqref{eq:di-is-a-metric} and approximate each $f_{\vert_{K_i}}$ by a partition of unity for $K_i$. To guarantee that the functions in the partition of unity are themselves in $\Lip_c^{\cK}(\G)$, it is important that each $K_i$ is open, so that we may simply extend them by zero. For this, conditions (3) and (6) are important. Furthermore, 
	we will
	sometimes 
	want to
	to treat functions with support contained in $\Go$ separately.
	As such, we will find condition (4) useful. 
	Lastly, condition (5) is included to guarantee $*$-invariance of the seminorm we define below, see \cref{def:total-seminorm} and \cref{lemma:LK-selfadjoint}. 
	
	\begin{remark}\label{remark:stratifications-degenerate-cases}
		\begin{enumerate}
			\item 	Suppose $\G = \Go$ is just a compact Hausdorff space, 
			and $d$ is a metric on $\Go$ inducing the topology.
			Considering 
			conditions (2) and (4)
			of \cref{def:metric-stratification}, we see that the only metric stratification of $\G$ available is $\cK= \{ K_e\}$ with $K_e = \Go$.
			\item Suppose $\G = \Gamma$ is a countable discrete group, and let $\ell \colon \Gamma \to [0,\infty)$ be a proper length function. Then $\Go = \{e\}$, where $e \in \Gamma$ is the unit, and there is of course a unique metric $d$ on $\Go$.  
			Considering condition (6) of \cref{def:metric-stratification} we see that the only metric stratification available is $\cK= (K_g)_{g\in \Gamma}$ with $K_g = \{g\} \times \Go = \{g\} \times \{e\}$. 
		\end{enumerate}
	
	\end{remark}
	
	In general, the choice of $\cK$ will depend on the situation. We note however, that under certain assumptions a particular choice of $\cK$ always exists. 
	
	\begin{lemma}\label{lemma:principal-metric-stratification}
		Let $\G$ be a principal \'etale groupoid, 
		where $\Go$ is compact, and let $d$ be a metric on $\Go$ inducing  the topology. 
		Suppose moreover that $\ell \colon \G \to [0,\infty)$ is a proper continuous  length function whose image $\Ima (\ell) \subseteq [0,\infty)$ is countable. Then $\cK = (K_s)_{s \in \Ima (\ell)}$ given by $K_s = \ell^{-1}(\{s\})$ is a metric stratification for $\G$. 
	\end{lemma}
	
	\begin{proof}
		By inspection we see that conditions (1)-(5) hold for $K_s = \ell^{-1}(\{s\})$ whenever $\ell$ is a proper continuous length function. To verify (6), note that by \cref{remark:comments-on-metric-stratification-def} it suffices to show that given any $s \in \Ima (\ell)$ the statement $d^{(s)}(\gamma , \mu) = 0$ implies $\gamma = \mu$ for $\gamma,\mu \in K_s$. But since $d$ is a metric on $\Go$, we see that $d^{(s)}(\gamma , \mu) = 0 $ implies $s(\gamma) = s(\mu)$ and $r(\gamma) = r(\mu)$, from which $\gamma = \mu$ follows by principality. 
	\end{proof}

	Given a metric $d$ on the compact space $\Go$ 
	inducing the topology,  
	we now fix a metric stratification $\cK = (K_i)_{i \in I}$ of $\G$ with respect to $d$, and define for every $i \in I$ and $f \in C_c(\G)$
		\begin{align*}
		L_{\mathrm{Lip}}^{K_i}(f) := \sup \{ \frac{\vert f(\gamma) - f(\mu) \vert}{d^{(i)}(\gamma, \mu)} \mid \gamma,\mu \in K_i   \}
	\end{align*}
	where $d^{(i)}$ is the metric from \eqref{eq:di-is-a-metric}.  $L_{\mathrm{Lip}}^{K_i}$ therefore measures the Lipschitz constant of $f$ restricted to $K_i$ for the metric induced by the range and source maps through $d^{(i)}$, and is therefore 
	sensitive to our particular choice of metric stratification $\cK$. 
	We then further define
	\begin{align}\label{eq:def-decomposition-Lipschitz-norm}
		L^{\cK}_{\mathrm{Lip}}(f) :=  \sup_{i \in I}   L^{K_i}_{\mathrm{Lip}}(f) 
	\end{align}
	for $f \in C_c(\G)$. Without further assumptions, it is entirely possible that $L^{\cK}_{\mathrm{Lip}}(f) = \infty$ for $f \in C_c(\G)$, and indeed we will soon restrict $C_c(\G)$ to a sup-operator system in order to construct compact quantum metric spaces. We first prove the following easy observation.

	\begin{lemma}\label{lemma:LK-selfadjoint}
		For every $f \in C_c (\G)$, we have $L^{\cK}_{\mathrm{Lip}}(f) = L^{\cK}_{\mathrm{Lip}}(f^*)$. 
	\end{lemma}
	
	\begin{proof}
		Fix $i \in I$, and denote by $d^{(i^{-1})}$ the metric on $K_i^{-1}$ from \eqref{eq:di-is-a-metric}. Importantly, note that for $\gamma, \mu \in K_i$, we have $d^{(i)}(\gamma, \mu) = d^{(i^{-1})}(\gamma^{-1}, \mu^{-1})$. We calculate
		\begin{align*}
			L_\mathrm{Lip}^{K_i}(f^*) &= \sup \{ \frac{\vert f^*(\gamma) - f^*(\mu)\vert}{d^{(i)}(\gamma, \mu)} \mid \gamma, \mu \in K_i  \} \\
			&= \sup \{ \frac{\vert \overline{f(\gamma^{-1})} - \overline{f(\mu^{-1})}\vert}{d^{(i)}(\gamma, \mu)} \mid \gamma, \mu \in K_i   \} \\
			&= \sup \{ \frac{\vert f(\gamma^{-1}) - f(\mu^{-1})\vert}{d^{(i)}(\gamma, \mu)} \mid \gamma, \mu \in K_i   \} \\
			&= \sup \{ \frac{\vert f(\gamma^{-1}) - f(\mu^{-1})\vert}{d^{(i^{-1})}(\gamma^{-1}, \mu^{-1})} \mid \gamma, \mu \in K_i  \} \\
			&= \sup \{ \frac{\vert f(\tilde{\gamma}) - f(\tilde{\gamma})\vert}{d^{(i^{-1})}(\tilde{\gamma}, \tilde{\mu})} \mid \tilde{\gamma}, \tilde{\mu} \in K_i^{-1}  \} \\
			&= L^{K_i^{-1}}_{\mathrm{Lip}}(f).
		\end{align*}
		Since $L^{\cK}_{\mathrm{Lip}}$ is defined in terms of the supremum over all $i \in I$, and $K_i \in \cK$ implies $K_i^{-1} \in \cK$, 
		the result follows. 
	\end{proof}
	
	In alignment with the setup of \cref{sec:prelims-CQMS}, we proceed to specify a  sub-operator system $\cX$ of $\Cred(\G)$ for which  $L^{\cK}_{\mathrm{Lip}}(f) < \infty$ for all $f \in \cX$. We define
	\begin{align}\label{eq:def-Lip-operator-system}
	\Lip_c^{\cK}(\G) := \{ f \in C_c (\G) \mid L^{\cK}_{\mathrm{Lip}}(f) < \infty  \}
	\end{align}
	where we have made explicit the fact that this operator system depends on the metric stratification $\cK$ of $\G$. Note that since $\Go = K_e \in \cK$, the characteristic function $1_{\Go}$ is in 
	$\Lip_c^{\cK}(\G)$. 
	Combined with \cref{lemma:LK-selfadjoint} we deduce that 
	$\Lip_c^{\cK}(\G)$
	is in fact a sub-operator system of $C_c(\G)$. 
	
	\begin{definition}\label{def:total-seminorm}
		Let $\G$ be an \'etale groupoid with compact unit space $\Go$, 
		and suppose $d$ is a metric on $\Go$ inducing the topology.
		Fix a metric stratification $\cK = (K_i)_{i\in I}$ for $\G$ with respect to $d$. 
		Suppose moreover that $\ell \colon \Go \to [0,\infty)$ is a proper continuous length function.  For every $n \geq 1$ we define a \emph{total seminorm} given by
			\begin{align}\label{eq:def-total-seminorm}
			L^{\cK, n} (f) := \max \{ L^n_\ell(f) , L_{\mathrm{Lip}}^{\cK} (f)  \},
		\end{align}
		for 
		$ f \in \Lip_c^{\cK}(\G)$ 
		(cf. \eqref{eq:def-Lip-operator-system}), where 
		\begin{align*}
			L^{n}_\ell(f) = \Vert \delta^n (f) \Vert_{\mathcal{L}_{C(\Go)}(\cE)}
		\end{align*} 
		as in \eqref{eq:definition-commutator-seminorm}, and
		\begin{align*}
			L^{\cK}_{\rm Lip}(f) = \sup_{i \in I}   L^{K_i}_{\mathrm{Lip}}(f) = \sup_{i \in I} \sup \{ \frac{\vert f(\gamma) - f(\mu) \vert}{d^{(i)}(\gamma, \mu)} \mid \gamma,\mu \in K_i   \}
		\end{align*}
		as in \eqref{eq:def-decomposition-Lipschitz-norm}.
	\end{definition}

	Note that by \cref{prop:iterated-commutator-adjointable} and \eqref{eq:def-Lip-operator-system}, we have $L^{\cK, n}(f) < \infty$ for all 
	$f \in \Lip_c^{\cK}(\G)$. 
	We therefore wish to study the following question.
	\begin{question}\label{question:when-do-we-get-cmqs}
		Is 
		$(\Lip_c^{\cK}(\G), L^{\cK, n})$
		a compact quantum metric space?
	\end{question}
	
	\begin{remark}\label{remark:metric-strat-vs-transf-gpds}
		Note that for a transformation groupoid $\G = \Gamma \ltimes X $, the seminorm in \eqref{eq:def-total-seminorm} for the case $n=1$ and when using the metric stratification $(K_g)_{g \in \Gamma}$ from \cref{remark:stratifications-degenerate-cases}, 
		bears a close resemblance to the seminorm defined in \cite{AustadKaadKyed2025} for crossed products using discrete groups. 
		However, there are minor differences. Let $(X,d)$ be a compact metric space. The ``horizontal part'' of the seminorm in \cite[Theorem A]{AustadKaadKyed2025}, will, for a crossed product $C(X) \rtimes \Gamma = C_0 (\Gamma , C(X))$ take the form $\max \{ L_d(f) , L_d (f^*)  \}$, where
		\begin{align*}
			L_d (f) = \sup_{g \in \Gamma} \{ \frac{\vert f(g,x) - f(g,y) \vert }{d(x,y)} \mid x \neq y   \} .
		\end{align*}
		Due to the presence of the $\max$ in the definition of $d^{(i)}$ in \eqref{eq:di-is-a-metric}, we may observe that in general $L_{\rm Lip}^{\cK} (f) \leq \max \{ L_d(f) , L_d (f^*)  \}$. Modifications to the seminorm $L_{\rm Lip}^{\cK}$ would make it a true generalization of the horizontal seminorm from \cite{AustadKaadKyed2025}. However, in light of \cref{remark:smaller-seminorm-CQMS-implies-CQMS}, and because we believe the resulting modified seminorm would obfuscate some arguments in the sequel, we opt to use $L^{\cK}_{\rm Lip}$ from \eqref{eq:def-decomposition-Lipschitz-norm}.  
	\end{remark}
	
	\begin{remark}
		Instead of the seminorm defined in \eqref{eq:def-total-seminorm}, another natural choice would be
		\begin{align*}
			\tilde{L}^{\cK,n}(f) = L^n_\ell(f) 
			+ L_{\mathrm{Lip}}^{\cK}(f).
		\end{align*}
		Note however that $\tilde{L}^{\cK,n}(f)$ and $L^{\cK,n}(f)$ are equivalent seminorms with $L^{\cK,n}(f) \leq  \tilde{L}^{\cK,n}(f) \leq 2 \cdot L^{\cK,n}(f)$ for all 
		$f \in \Lip_c^{\cK}(\G)$, 
		and so they induce equivalent metrics and topologies on  
		$S(\Lip_c^{\cK}(\G))$
		through the Monge-Kantorovi\v{c} metric construction from \eqref{eq:def-monge-kantorovic-metric}.
	\end{remark}

	\subsection{A characterization of metric groupoids yielding compact quantum metric spaces}
	The main result of this section is \cref{thm:new-characterization}, which will provide  a sufficient condition for \cref{question:when-do-we-get-cmqs} to have a positive answer. It does so in terms existence of Fourier multipliers which are compatible with the metric stratification $\cK$ in a sense we 
	now make precise.

	\begin{definition}\label{def:K-continuous-multipliers}
		Suppose $\G$ is an \'etale groupoid with compact unit space $\Go$, and suppose  $d$ is a metric inducing the topology on $\Go$. 
		Fix  a metric stratification $\cK$ of $\G$ with respect to $d$, and 
		let $\phi \in \FS(\G)$. 
		If there exists a real number $C_\phi \geq 0$ for which
		\begin{align*}
			L_\mathrm{Lip}^\cK (m_\phi (f)) \leq C_\phi L_\mathrm{Lip}^\cK(f)
		\end{align*}
		for all 
		$f \in \Lip_c^{\cK}(\G)$,
		we say that $\phi$ is \emph{$\cK$-continuous with coefficient $C_\phi$}. If the coefficient is not important, we just say that $\phi$ is \emph{$\cK$-continuous}.
	\end{definition}
	
	We may now establish the following result, which tells us that if a Fourier multiplier is $\cK$-continuous, then 
	it
	is in fact continuous with respect to the total seminorm $L^{\cK, n}$ for any $n\geq 1$. 

	\begin{lemma}\label{lemma:multipliers-extend-to-Roe}
		Let $\G$ be an \'etale groupoid with compact unit space $\Go$, and let $\ell \colon \G \to [0,\infty)$ be a proper continuous length function. 
		Suppose moreover $d$ is a metric on $\Go$ inducing its topology,
		and fix a metric stratification $\cK$ of $\G$.
		Let $f \in C_c (\G)$ and $\phi \in \FS(\G)$, 
		and denote by $T^\phi$ 
		the extension of $m_\phi$ to $\Cred ( \G \ltimes \beta\G)$ given by \eqref{eq:Roe-algebra-lifted-multiplier}. 
		Then for all $n \geq 1$
		\begin{align*}
			\delta^n(m_\phi (f)) = T^\phi (\delta^n(f)).
		\end{align*}
		Denote by $\Vert T^\phi\Vert$ the multiplier norm of $T^\phi$. 
		If $\phi$ is $\cK$-continuous with coefficient $C_\phi \geq 0$ 
		we immediately get
		\begin{align*}
			L^{\cK,n} (m_\phi (f)) \leq \max \{ \Vert T^\phi \Vert , C_\phi  \} \cdot L^{\cK,n}(f)
		\end{align*}
		for all 
		$f \in \Lip_c^{\cK}(\G)$. 
	\end{lemma}
	
	\begin{proof}
		Note first that 
		$\delta^n(f) \in \Cred ( \G \ltimes \beta\G)$
		by \cref{prop:iterated-commutator-adjointable}, The multiplier $T^\phi$ exists by \cref{remark:existence-and-boundedness-of-Tphi} and acts as
		\begin{align*}
			(T^\phi F)(\gamma, \mu) = \phi(\gamma) F(\gamma, \mu)
		\end{align*}
		for $(\gamma, \mu) \in \G \tensor[_s]{*}{_r} \G$ and  $F \in \Cred( \G \ltimes \beta \G)$. 
		Let now 
		$F(\gamma, \mu) = (\ell(\gamma \mu) - \ell(\mu))^n f(\gamma)$. 
		We saw in the proof of \cref{prop:iterated-commutator-adjointable} that we may view  
		$F \in \Cred( \G \ltimes \beta \G) \subseteq \curlyL_{C(\Go)}(\E)$. 
		By \eqref{eq:action-of-Roe-alg-on-E} it follows that
		\begin{align*}
			\tilde{\Lambda}(\delta^n (m_\phi (f)))(\xi) (\gamma) &= \sum_{\mu \in \G_{s(\gamma)}} (\ell(\gamma) - \ell(\mu))^n (m_\phi (f))(\gamma \mu^{-1}) \xi(\mu) \\
			&= \sum_{\mu \in \G_{s(\gamma)}} \phi(\gamma \mu^{-1}) (\ell(\gamma) - \ell(\mu))^n  f(\gamma \mu^{-1}) \xi(\mu) \\
			&= \tilde{\Lambda}(T^\phi(F)) (\xi)(\gamma)
		\end{align*}
		for all $\xi \in \E$ and all $\gamma \in \G$.
		Thus for all 
		$f \in \Lip_c^{\cK}(\G)$
		\begin{align*}
			L_\ell^n(m_\phi (f)) &= \Vert \delta^n (m_\phi (f))\Vert_{\curlyL_{C(\Go)}(\E)} \
			= \Vert T^\phi(F) \Vert_{\Cred(\G \ltimes \beta \G)} \\ 
			&\leq \Vert T^\phi \Vert \Vert F \Vert_{\Cred(\G \ltimes \beta \G)} 
			= \Vert T^\phi \Vert \Vert \delta^n(f) \Vert_{\curlyL_{C(\Go)}(\E)} = \Vert T^\phi \Vert L_\ell^n (f).
		\end{align*}
		We combine this with the assumption $L_{\mathrm{Lip}}^{\cK} (m_\phi (f)) \leq C_\phi L_{\mathrm{Lip}}^{\cK} ( f)$ for all 
		$f \in \Lip_c^{\cK}(\G)$
		to obtain 
		\begin{align*}
			L^{\cK,n} (m_\phi (f)) \leq \max \{ \Vert T^\phi \Vert , C_\phi  \} \cdot L^{\cK,n}(f)
		\end{align*}
		for all 
		$f \in \Lip_c^{\cK}(\G)$, 
		which finishes the proof. 
	\end{proof}

	To give our characterization below, we will make use of \cref{prop:Ozawa-Rieffel-characterization}. It is therefore useful to specify what the set 
	$E^\sigma_L$
	looks like for our candidate compact quantum metric space 
	$(\Lip_c^{\cK}(\G), L^{\cK,n})$.
	We first define 
	\begin{align}\label{eq:def-E-set}
		E_{\cK, n} := \{ f \in \Lip_c^\cK (G) \mid L^{\cK, n}(f) \leq 1\}.
	\end{align}
	Now, 
	fix
	any probability measure $\mu$ on $\Go$ and resulting state $\sigma=\phi_\mu = \overline{\mu} \circ E  \in S(\Cred(\G))$ as in \cref{sec:prelims-groupoids}. Then
	the restriction of 
	$\sigma$ 
	is 
	a state on
	$\Lip_c^{\cK}(\G)$
	and 
	$E^\sigma_L$
	takes the  form 
	\begin{align}\label{eq:BL-for-groupoids}
	E^\sigma_{\cK, n}: = \{ f \in E_{\cK, n} \mid \sigma(f) = 0  \}.
	\end{align}
	In the sequel, we will want to cut 
	$ E^\sigma_{\cK, n}$
	down to a smaller set using Fourier multipliers. More specifically, we will want to focus on the functions supported on elements of length smaller than some $t \geq 0$. We set
	\begin{equation}\label{eq:BLt-for-groupoids}
		\begin{split}
		E^\sigma_{\cK, n}[t] :&= \{  f\in \Lip_c^{\cK}(\G) \mid \sigma (f) = 0, L^{\cK, n}(f) \leq 1 \text{ and $f(\gamma) = 0$ for $\ell(\gamma) > t$}    \}\\
		&= \{ f \in E^\sigma_{\cK, n} \mid f(\gamma) = 0 \text{ for $\ell(\gamma)> t$}\}.
		\end{split}
	\end{equation}
	We record the following lemma, which will be of importance in the proof of \cref{thm:new-characterization}, and which could also be of independent interest. 
	
	\begin{lemma}\label{lemma:compact-subsets-yield-CQMS}
		Let $\G$ be an \'etale groupoid with compact unit space $\Go$, and let $\ell \colon \G \to [0,\infty)$ be a proper continuous length function. 
		Suppose moreover $d$ is a metric on $\Go$ inducing its topology, and fix a metric stratification $\cK$ of $\G$.
		Let 
		$\cX \subseteq  \Lip_c^{\cK}(\G)$
		be a sub-operator system such that there is $t \geq 0$ for which $\supp f \subseteq B_\ell(t)$ for all $f \in \cX$. 
		Then $(\cX, (L^{\cK, n})_{\vert_{\cX}})$ is a compact quantum metric space for all $n\in \N$. In particular, if $\G$ is compact, 
		then 
		$( \Lip_c^{\cK}(\G), L^{\cK, n})$
		is a compact quantum metric space for any $n \in \N$. 
	\end{lemma}

	\begin{proof}
		Fix a state $\sigma \in S(\Cred(\G))$ be given by integration against a probability measure $\mu$ on $\Go$, that is, $\sigma =  \overline{\mu} \circ E$. 
		This restricts to a state on $\cX$ which we also denote by  
		$\sigma$.
		To show the first statement it suffices by  \cref{prop:Ozawa-Rieffel-characterization} to show that the set 
		\begin{align*}
			\{ f \in \cX \mid \sigma(f) = 0 \text{ and } L^{\cK,n}(f) \leq 1  \}
		\end{align*}
		is totally bounded in norm. 
		Since this set is contained in 
		$E^\sigma_{\cK,n}[t]$ from \eqref{eq:BLt-for-groupoids} 
		by assumption on $\cX$, we prove that 
		$E^\sigma_{\cK,n}[t]$
		is totally bounded, from which the first statement will follow.  
		Let therefore $\varepsilon > 0$ be given. We wish to show that 
		$E^\sigma_{\cK,n}[t]$
		can be covered by finitely many $\varepsilon$-balls in the $\Cred(\G)$-norm. 
		
		Note first that $B_\ell(t)$ is compact since $\ell$ is proper. 
		Find a cover of $B_\ell(t)$ by sets from $\cK$, which by compactness of $B_\ell(t)$ has a finite subcover $B_\ell(t) \subseteq \coprod_{i=1}^m K_i$, after reindexing if necessary. Importantly, since the sets in $\cK$ are disjoint and $B_\ell(t)^{-1} = B_\ell(t)$, we have that if $K_i$ is in the subcover, so is $K_i^{-1}$. 
		Let $d^{(i)}$ be the metric induced on $K_i$ through \eqref{eq:di-is-a-metric} so that $(K_i, d^{(i)})$ is a totally bounded metric space for every $i$. 
		Then, for every $i=1,\ldots , m$ we may find finitely many points $\gamma_1^{(i)}, \ldots , \gamma_{h_i}^{(i)}$ for which $K_i \subseteq \cup_{j=1}^{h_i} B_{d^{(i)}}(\gamma_j^{(i)}, \varepsilon)^\circ$, where $B_{d^{(i)}}(\gamma_j^{(i)}, \varepsilon)^\circ$ is the open $\varepsilon$-ball around $\gamma_j^{(i)}$ in the metric $d^{(i)}$.

		Now, for each $i = 1, \ldots , m$, find a partition of unity subordinate to $(B_{d_i}(\gamma_j^{(i)}, \varepsilon)^\circ)_{j=1}^{h_i}$, say $(\rho_j^{(i)})_{i=1}^{h_i}$. We may arrange that all the $\rho_j^{(i)}$ are Lipschitz continuous with respect to the metrics $d^{(i)}$ \cite[Theorem 2.6.5]{CobzasMiculescuNicolae-LipschitzFunctions2019}.   
		Note moreover that as $\rho_j^{(i)}$ is a Lipschitz continuous  function on $K_i$ for each $i,j$, and subordinate to open subsets of $K_i \subseteq \G$, 
		the extension by zero of $\rho_j^{(i)}$ can be viewed as an element of 
		$\Lip_c^{\cK}(\G)$ 
		for each $i=1, \ldots ,m$ and $j = 1, \ldots, h_i$. 
		
		Now, for $f \in E^\sigma_{\cK, n}[t]$, write $f = \sum_{i=1}^m f_{\vert_{K_i}}$, which we may do uniquely as the $K_i$ are disjoint, and set
		\begin{align*}
		\Phi_\varepsilon^{(i)} (f) 
			:= \sum_{j=1}^{h_i} f_{\vert_{K_i}} (\gamma_j^{(i)}) \rho_j^{(i)} = \sum_{j=1}^{h_i} f (\gamma_j^{(i)}) \rho_j^{(i)}.
		\end{align*}
		We then have a finite-dimensional approximation of $f$ through
		\begin{align*}
			f \approx \Phi_\varepsilon(f) := 
		\sum_{i=1}^m \Phi_\varepsilon^{(i)} (f) 
			=  \sum_{i=1}^m \sum_{j=1}^{h_i} f (\gamma_j^{(i)}) \rho_j^{(i)}.
		\end{align*}
		Recall the definition of the $I$-norm from \eqref{eq:def-I-norm} and that it dominates the reduced $C^*$-algebra norm. 
		Note that since we are considering finitely many $K_i$ and $\max_{i\in \{1, \ldots ,n\}} \sup_{\gamma \in K_i} \ell(\gamma) < \infty$, \cref{assumption:uniform-fiber-cardinality-bound} yields that there must exist $F >0$ which is an upper bound on the quantities $\vert\{x \in \G_u \cap K_i\}\vert$ and $\vert\{x \in \G^u \cap K_i\}\vert$ for $i = 1, \ldots m$.
		Now, as $K_e$ is in the subcover, we  
		calculate
		\begin{align*}
			\Vert &f - \sum_{i=1}^m \sum_{j=1}^{h_i} f(\gamma_j^{(i)}) \rho_j^{(i)} \Vert = \Vert \sum_{i=1}^m f_{\vert_{K_i}} - \sum_{i=1}^m \sum_{j=1}^{h_i} f(\gamma_j^{(i)}) \rho_j^{(i)} \Vert \\
			&\leq \sum_{i=1}^m \Vert f_{\vert_{K_i}} - \sum_{j=1}^{h_i} f_{\vert_{K_i}} (\gamma_j^{(i)}) \rho_j^{(i)} \Vert \\
			&\leq \sum_{i=1}^m \Vert f_{\vert_{K_i}} - \sum_{j=1}^{h_i} f_{\vert_{K_i}} (\gamma_j^{(i)}) \rho_j^{(i)} \Vert_I \\
			&= \sum_{i=1}^m \max \{ \sup_{u \in \Go}  \sum_{\mu \in \G_u} \big\vert f_{\vert_{K_i}} (\mu) - \sum_{j=1}^{h_i} f_{\vert_{K_i}} (\gamma_j^{(i)}) \rho_j^{(i)}(\mu)\big\vert, \sup_{u \in \Go}  \sum_{\mu \in \G^u} \big\vert f_{\vert_{K_i}} (\mu) - \sum_{j=1}^{h_i} f_{\vert_{K_i}} (\gamma_j^{(i)}) \rho_j^{(i)}(\mu)\big\vert \} \\
			&\leq \sum_{i=1}^m \max \{  \sup_{u \in \Go}  \sum_{\mu \in \G_u} \sum_{j=1}^{h_i}\big\vert f_{\vert_{K_i}} (\mu) -  f_{\vert_{K_i}} (\gamma_j^{(i)})\big\vert \rho_j^{(i)}(\mu) ,  \sup_{u \in \Go}  \sum_{\mu \in \G^u} \sum_{j=1}^{h_i}\big\vert f_{\vert_{K_i}} (\mu) -  f_{\vert_{K_i}} (\gamma_j^{(i)})\big\vert \rho_j^{(i)}(\mu) \} \\
			&\leq \sum_{i=1}^m \max \{ \sup_{u \in \Go}  \sum_{\mu \in \G_u\cap K_i} \sum_{j=1}^{h_i} d^{(i)} (\mu, \gamma_j^{(i)}) L^{\cK}_{\mathrm{Lip}}(f_{\vert_{K_i}}) \rho_j^{(i)}(\mu)  , \sup_{u \in \Go}  \sum_{\mu \in \G^u\cap K_i} \sum_{j=1}^{h_i} d^{(i)} (\mu, \gamma_j^{(i)}) L^{\cK}_{\mathrm{Lip}}(f_{\vert_{K_i}}) \rho_j^{(i)}(\mu)     \} \\
			&=\sum_{i=1}^m \max \{ \sup_{u \in \Go}  \sum_{\mu \in \G_u\cap K_i}  \varepsilon \cdot L^{\cK}_{\mathrm{Lip}}(f_{\vert_{K_i}})   , \sup_{u \in \Go}  \sum_{\mu \in \G^u\cap K_i}  \varepsilon \cdot L^{\cK}_{\mathrm{Lip}}(f_{\vert_{K_i}})   \} \\
			&\leq \sum_{i=1}^m \max \{   F  \varepsilon \cdot L^{\cK}_{\mathrm{Lip}}(f_{\vert_{K_i}})   , F  \varepsilon \cdot L^{\cK}_{\mathrm{Lip}}(f_{\vert_{K_i}})   \} \\
			&=  m F  \varepsilon \cdot L^{\cK}_{\mathrm{Lip}}(f_{\vert_{K_i}})   ,  \\
			&\leq m F \varepsilon L^{\cK,n}(f).
		\end{align*}
		Thus we may arrange that 
		$E^\sigma_{\cK,n}[t]$
		is less than $\varepsilon$ away in $\Cred(\G)$-norm from its image under $ \Phi_\varepsilon$ in the finite-dimensional subspace of 
		$\Cred(\G)$ 
		spanned by $((\rho_j^{(i)})_{j=1}^{h_i})_{i=1}^m$. We proceed to show that $\Ima (\Phi_\varepsilon)$ is bounded in operator norm. 
		
		For this purpose, let 
		$f \in E^\sigma_{\cK,n}[t]$, 
		and 
		\begin{align*}
			\Phi_\varepsilon(f) = \sum_{i=1}^m \sum_{j=1}^{h_i} f(\gamma_j^{(i)}) \rho_j^{(i)}.
		\end{align*}
		We will show there is a uniform upper bound on the reduced $C^*$-algebra norm on such elements. First, we obtain the upper bound 		
		\begin{align*}
			\Vert \Phi_\varepsilon (f) \Vert \leq \sum_{i=1}^m \sum_{j=1}^{h_i} \vert f(\gamma_j^{(i)}) \vert \Vert \rho_j^{(i)} \Vert 
			\leq \max_{i,j} \{ \Vert \rho_j^{(i)}\Vert \} \cdot \sum_{i=1}^m \sum_{j=1}^{h_i} \vert f(\gamma_j^{(i)})\vert,
		\end{align*}
		where $\max_{i,j} \{ \Vert \rho_j^{(i)}\Vert \}$ is certainly finite as there are only finitely many $\rho_j^{(i)}$. The value is independent of $f$. We can obtain a further bound by noting that  
		\begin{align*}
			\sum_{i=1}^m \sum_{j=1}^{h_i} \vert f(\gamma_j^{(i)})\vert \leq m \cdot \max\{ h_i \mid i=1,\ldots,m \} \cdot \max \{ \vert f(\gamma_j^{(i)}) \},
		\end{align*}
		where only the last factor depends on $f$. As such, it will suffice to show that $\vert f (\gamma_j^{(i)}) \vert $ is uniformly bounded above for all $f \in E_{\cK, n}[t]$ and $i= 1,\ldots, m$, $j=1,\ldots, h_i$.

		We first consider the case $i \neq e$. Let $A_n > 0$ be a real number which bounds the $\infty$-norm by the $\ell^n$-weighted $2$-norm for functions with compact support away from the unit space (recall that $\ell$ may take values in $(0,1)$). That is, $\Vert f \Vert_\infty \leq A_n \cdot \sup_{u\in \Go} \Vert f \cdot \ell^n \Vert_{\ell^2(\G_u)}$ for all $f \in C_c (\G)$ with $\supp(f) \subseteq \G \setminus \Go$.  
		We then calculate 
		\begin{align*}
			\vert f(\gamma_j^{(i)}) \vert &\leq \Vert f\vert_{K_i} \Vert_\infty \leq A_n \cdot \sup_{u \in \Go} \Vert f\vert_{K_i} \cdot \ell^n \Vert_{\ell^2(\G_u)} \\
			&\leq A_n \cdot \sup_{u \in \Go} \Vert f \cdot \ell^n \Vert_{\ell^2(\G_u)}
			= A_n \cdot \Vert \delta^n (f) 1_{\Go} \Vert_{\E} \\
			&\leq A_n\cdot \Vert \delta^n(f) \Vert_{\curlyL_{C(\Go)}(\E)} 
			\leq A_n \cdot L^{\cK,n}(f). 
		\end{align*}  
		Since  $L^{\cK,n}(f)$ is uniformly bounded for 
		$f \in E^\sigma_{\cK,n}[t]$
		this covers the case $i \neq e$.
		
		If $i=e$, we take a different approach. Since 
		$f \in E^\sigma_{\cK,n}[t]$
		is subject to 
		$\sigma(f) = \overline{\mu} \circ E(f) = \overline{\mu}(f_{\vert_{\Go}}) = 0$ 
		and $f\vert_{\Go} \in C(\Go)$, we know that $f\vert_{\Go}(u) = 0$ for some $u \in \Go$. We may therefore calculate an upper bound on the supremum norm of $f\vert_{\Go} = f_{\vert_{K_e}}$ as follows
		\begin{align*}
			\Vert f\vert_{\Go} \Vert_\infty \leq \mathrm{diam}(\Go, d) \cdot L^{K_e}_{\mathrm{Lip}}(f \vert_{\Go}) \leq \mathrm{diam}(\Go, d)  \cdot L^{\cK, n}(f),
		\end{align*} 
		where $\mathrm{diam}(\Go, d)$ is the diameter of the compact metric space $(\Go, d)$. 
		There is therefore a uniform upper bound on the values $\vert f (\gamma_j^{(i)}) \vert $, from which we deduce that $\Ima( \Phi_\varepsilon)$ is a bounded set in operator norm. Since it spans a finite-dimensional subspace of $\Cred (\G)$, we deduce that $\Ima( \Phi_\varepsilon)$ is totally bounded. The first part of the lemma then follows. 
		
		The statement that for $\G$ compact, 
		$( \Lip_c^{\cK}(\G), L^{\cK, n})$
		is a compact quantum metric space for every $n \geq 1$ follows easily by the first part.  
		Since $\ell$ is bounded on $\G$, we observe that 
		$E^\sigma_{\cK, n} \subseteq E^\sigma_{\cK, n}[t]$
		for $t$ large enough, for any $n \geq 1$. The conclusion now follows by the first part of the lemma.
	\end{proof}

	We are now in a position to prove the following result, which 
	constitutes the first main theorem of the article. 

	\begin{theorem}\label{thm:new-characterization}
		Let $\G$ be an \'etale groupoid with compact unit space $\Go$, and let $\ell \colon \G \to [0,\infty)$ be a proper continuous length function. 
		Suppose moreover $d$ is a metric on $\Go$ inducing its topology, and fix a metric stratification $\cK$ of $\G$.
		Let $E_{\cK, n}$ be as in \eqref{eq:def-E-set}. 
		
		Consider the statements
		\begin{enumerate}
			\item 
			For every $\varepsilon > 0$ there is 
			$\phi \in C_c(\G)$ 
			such that $m_\phi$ is unital and $\cK$-continuous, and such that
			\begin{align}\label{eq:norm-approximation-yields-CQMS}
				\sup_{f\in E_{\cK, n}} \Vert f - m_{\phi}(f) \Vert_{\Cred(\G)} < \varepsilon.
			\end{align}
			\item $( \Lip_c^{\cK}(\G), L^{\cK,n})$  
			is a compact quantum metric space.
		\end{enumerate}
		Then (1) implies (2).

		Moreover, the converse implication holds if $\G$ admits a sequence of functions
		$(\phi_j)_{j \in \N} \subseteq C_c(\G)$ converging uniformly to $1$ on compact subsets, 
		such that $m_{\phi_j}$ is unital and $\cK$-continuous with coefficient 
		$D_j \geq 0$
		for all $j$, and
		satisfying $\sup_j \Vert m_{\phi_j} \Vert < \infty$.
		
	\end{theorem}

	\begin{proof}
		Throughout the proof we fix a state $\sigma \in S(\Cred(G))$, and we denote the restriction to  $\Lip_c^{\cK}(\G)$ by $\sigma$ as well. 
		
		Suppose first (1) holds. In order to verify that $( \Lip_c^{\cK}(\G), L^{\cK,n})$ is a compact quantum metric space, it suffices to show that the set $E^\sigma_{\cK, n}$ defined as in \eqref{eq:BL-for-groupoids} is totally bounded. 
		Let therefore $\varepsilon > 0$, and choose 
		$\phi \in C_c(\G)$ 
		such that
		$m_\phi$ is unital and $\cK$-continuous, and such that
		\begin{align*}
			\sup_{f\in E_{\cK, n}} \Vert f - m_{\phi}(f) \Vert < \varepsilon.
		\end{align*}
		Combining the assumption that $\phi$ is $\cK$-continuous with \cref{lemma:multipliers-extend-to-Roe}, it follows that there exists $D >0$ such that $L^{\cK,n}(m_\phi f) \leq D \cdot L^{\cK,n}(f)$ for all 
		$f \in \Lip_c^{\cK}(\G)$. 
		As $\phi \in C_c(\G)$, there is $t \geq 0$ large enough so that $\supp(\phi) \subseteq B_\ell(t)$, and therefore in turn $\supp (m_\phi (f)) \subseteq B_\ell(t)$ for all 
		$f \in E^\sigma_{\cK, n}$.
		Since  
		$m_\phi$ is unital, 
		these observations together combine
		to guarantee 
		$m_{\phi}(f) \in D \cdot E^\sigma_{\cK, n}[t]$ for any $f \in E^\sigma_{\cK, n}$.
		The set 
		$E^\sigma_{\cK, n}[t]$
		is totally bounded by \cref{lemma:compact-subsets-yield-CQMS},  
		and therefore so is 
		$D \cdot E^\sigma_{\cK, n}[t]$.
		Thus the assumption $\sup_{f\in E_{\cK,n}} \Vert f - m_{\phi}(f) \Vert < \varepsilon$ tells us that 
		$E^\sigma_{\cK, n}$ 
		is $\varepsilon$ away in norm from a norm-totally bounded set, and is therefore itself totally bounded in norm. It follows that (2) holds.

		Conversely, suppose (2) holds and that we have a sequence  $(\phi_j)_{j \in \N}$ as in the statement of the theorem.
		Let $\varepsilon > 0$ be arbitrary, and let $D \geq 1$ be such that  
		$\sup_j \Vert m_{\phi_j} \Vert \leq D$.
		By \cref{prop:Ozawa-Rieffel-characterization}, we may fix 
		$f_1, \ldots f_l \in E^\sigma_{\cK, n}$
		such that for all 
		$f \in E^\sigma_{\cK, n}$ 
		there exists $k \leq l$ such that $\Vert f - f_k \Vert < \frac{\varepsilon}{4D}$. Since all the $f_k$ are compactly supported, there exists $M \in \N$ such that for all $j \geq M$ and for all $k =1, \ldots ,l$ we have
		\begin{align*}
			\Vert f_k - m_{\phi_j}(f_k) \Vert < \frac{\varepsilon}{4}. 
		\end{align*}
		Since $ \Vert m_{\phi_j} \Vert \leq D$ for all $j$, we also have
		\begin{align*}
			\Vert m_{\phi_j} (f_k) - m_{\phi_j}(f) \Vert \leq \Vert m_{\phi_j} \Vert \cdot \Vert f- f_k \Vert < \frac{\varepsilon}{4}.
		\end{align*}
		Now, given 
		$f \in E^\sigma_{\cK, n}$ 
		we find $k\leq l$ such that $\Vert f - f_k \Vert < \frac{\varepsilon}{4D}$. Then, for all $m \geq M$ we have
		\begin{align*}
			\Vert f - m_{\phi_j}(f) \Vert \leq \Vert f - f_k \Vert + \Vert f_k - m_{\phi_j}(f_k) \Vert + \Vert m_{\phi_j}(f_k) - m_{\phi_j}(f) \Vert < \frac{\varepsilon}{4D} + \frac{\varepsilon}{4} + \frac{\varepsilon}{4} < \varepsilon.
		\end{align*}
		We deduce 
		that
		\begin{align*}
			\sup_{f\in E^\sigma_{\cK, n}} \Vert f - m_{\phi}(f) \Vert_{\Cred(\G)} < \varepsilon.
		\end{align*}
		Now note that for any $ f \in E_{\cK, n}$ we have  $f - \sigma(f)\cdot 1_{\Go} \in E^\sigma_{\cK , n}$.
		By unitality of $m_{\phi_j}$ we may write 
		\begin{align*}
			\Vert f - m_\phi(f) \Vert_{\Cred(\G)} = \Vert f - \sigma(f) - m_\phi(f- \sigma(f)) \Vert_{\Cred(\G)} ,
		\end{align*}
		from which we see that (2)  follows. 
	\end{proof}

	\begin{remark}\label{remark:amenability-groupoid}
		By \cite[Proposition 2.2.13 and Proposition 2.2.7]{DelarocheRenaultAmenable2000} a second-counable \'etale groupoid $\G$ is amenable if and only if it admits a sequence $(h_j)_j$ of continuous positive definite functions with compact support such that
		\begin{enumerate}
			\item[(a)] $(h_j)_{\vert_{\Go} } = 1$, and
			\item[(b)] $\lim_j h_j = 1$ uniformly on every compact subset of $\G$. 
		\end{enumerate}
		Then the associated multipliers $m_{h_{j}}$ are unital completely positive maps by \cref{prop:pos-def-iff-cp-map}. Thus if $(h_j)_j$ are $\cK$-continuous for 
		for all $j$, we see that (2) implies (1) in \cref{thm:new-characterization}, and we may in fact use $(h_j)_j$ to verify \eqref{eq:norm-approximation-yields-CQMS}. 
		Indeed, 
		this will be implicit when we verify that AF groupoids give rise to compact quantum metric spaces in \cref{thm:CQMS-from-AF-groupoids}.  
	\end{remark}

		The characterization of compact quantum metric spaces arising from groupoids in \cref{thm:new-characterization} is new even in the case of discrete groups. Recall that for a discrete group $\Gamma$ there is a unique choice of metric stratification $\cK$ by \cref{remark:comments-on-metric-stratification-def}, and we see that $\cK$-continuity of $\phi \in C_c(\Gamma)$ is automatic. 
		A countable discrete group $\Gamma$  is said to be \emph{weakly amenable} if it admits a sequence $(\phi_j)_j$ such that $\phi_j \to 1$ uniformly on compact subsets and
		$\sup_j \Vert m_{\phi_j} \Vert_{\rm cb} < \infty$. 
		We may assume $\phi_j(e) = 1$, that is $m_{\phi_j}$ is unital, for all $j$, and will do so below. 
		Given a proper length function $\ell \colon \Gamma \to [0,\infty)$, we therefore see that \cref{thm:new-characterization} provides a sufficient and necessary condition for $(\Gamma, \ell)$ to yield a compact quantum metric space as in  \cref{example:CQMS-from-groups} whenever $\Gamma$ is weakly amenable. Notably, groups with polynomial growth are weakly amenable as they are amenable, and word-hyperbolic groups are weakly amenable by \cite{OzawaHyperbolicWeakAmenable2008}. We summarize these observations in the following result. 
		\begin{corollary}\label{cor:FD-CQMS-approx-discrete-groups}
			Suppose $\Gamma$ is a weakly amenable countable discrete group and  $(\phi_j)_j$ is a sequence as above. Let $L_\ell$ be the seminorm from \eqref{eq:slip-norm-length-discrete-group-case}.  Then $(C_c (\Gamma), L_\ell)$ is a compact quantum metric space if and only if  for any $\varepsilon > 0$ there is $M \in \N$ such that for any $j \geq M$ we have
			\begin{align}\label{eq:norm-approximation-discrete-group-case}
				\sup_{f \in E} \Vert f - m_{\phi_j}(f) \Vert_{\Cred(\Gamma)} < \varepsilon,
			\end{align}
			where 
			$E = \{ f\in C_c (\Gamma) \mid  L_\ell(f) \leq 1\}$. 
			In particular, \eqref{eq:norm-approximation-discrete-group-case} holds for finitely generated groups of polynomial growth and word-hyperbolic groups with their associated sequences $(\phi_j)_j$.

		\end{corollary}
		
		It is natural to compare \cref{cor:FD-CQMS-approx-discrete-groups} to to Kaad's characterization of compact quantum metric spaces from \cite[Theorem 3.1]{KaadExternal24}. In it, he characterizes compact quantum metric spaces as pairs $(\cX, L)$ consisting of  operator systems $\cX$ and  slip-norms $L \colon \cX \to [0,\infty)$ which admit positive finite-dimensional approximations in the following sense: For any $\varepsilon >0$, there exist an  operator system $\cY$, a unital isometry $\iota \colon \cX \to \cY$ and a unital positive map $\Phi \colon \cX \to \cY$ with finite-dimensional range, such that 
				\begin{align*}
						\Vert \iota(x) - \Phi(x) \Vert_{\cY} \leq \varepsilon \cdot L(x)
				\end{align*}
		for all $x \in \cX$. By comparison, \cref{cor:FD-CQMS-approx-discrete-groups} tells us that if the countable discrete group $\Gamma$ is weakly amenable with associated sequence $(\phi_j)_j$, then we obtain finite-dimensional approximations with $\cY= \cX = C_c(\Gamma)$, $\iota$ the identity map on $ C_c(\Gamma)$, and $\Phi = m_{\phi_j}$ for $j$ sufficiently large. Note however that to have the codomain $\cY$ equal $\cX = C_c(\Gamma)$, we have sacrificed positivity of $\Phi$ in exhange for complete boundedness. If $(\Gamma, \ell)$ has polynomial growth, then $\Gamma$ is amenable and the sequence $(\phi_j)_j$ can be chosen so that $\phi_j$ is positive definite and hence $m_{\phi_j}$ is completely positive for all $j$.

	Lastly in this section we prove a groupoid analogue of the result from \cite{ChristensenIvanRD}.  Specifically, we show that under the assumption of rapid decay, we may always find $n$ large enough so that when employing the seminorm from \eqref{eq:def-total-seminorm} using $n$ iterated commutators we obtain a compact quantum metric space. Note however that in the result below we still have to impose some compatibility between the metric stratification $\cK$ and the length function in order to guarantee that the Fourier multiplier is $\cK$-continuous. 
	We will also  impose the very weak condition on the length function that for any $t \in \ell(\G) \subseteq \R$ there is some $\delta_t>0$ such that $[t, t+\delta_t) \cap \ell(\G) = \{t\}$. This condition is of course satisfied in the case of integer-valued length functions. Note also that for discrete groups, this condition is implied by properness of the length function.

	\begin{proposition}\label{prop:Christensen-Ivan-for-groupoids}
		Let $\G$ be an \'etale groupoid with compact unit space $\Go$, and let $\ell \colon \G \to [0,\infty)$ be a proper continuous length function.
		Suppose further that for any $t \in \ell(\G) \subseteq \R$ there is some $\delta_t>0$ with $[t, t+\delta_t) \cap \ell(\G) = \{t\}$. 
		Furthermore, assume $d$ is a metric on $\Go$ inducing the topology. 
		Fix a metric stratification $\cK = (K_i)_{i\in I}$ with respect to $d$, and assume moreover that for each $i \in I$ there is $r_i \in \R$ such that $K_i \subseteq \ell^{-1}(\{r_i\})$.
		
		If $(G, \ell)$ has rapid decay with constants $C,s >0$, then  
		for any $n \geq \lfloor s \rfloor +1$ the pair 
		$( \Lip_c^{\cK}(\G), L^{\cK, n})$
		is a compact quantum metric space. 
		
		In particular, if $(G,\ell)$ has polynomial growth  
		bounded by some polynomial $p$ of degree $d$, then $( \Lip_c^{\cK}(\G), L^{\cK, n})$ is a compact quantum metric space for all $n > \lfloor d + \frac{1}{2}\rfloor +1$. 
	\end{proposition}
	\begin{remark}
		Note that in the statement of the proposition we are not requiring that for $i\neq j\in I$ we have $r_i \neq r_j$. So we could have several $K_i$ be subsets of the same $\ell^{-1}(\{r\})$, $r \in \R$. This is for example very natural when considering a transformation groupoid $\Gamma \ltimes X$, where $K_i = \{\gamma_i\} \times X$ and $K_j = \{\gamma_j\} \times X$ would be subsets of the same $\ell^{-1}(\{r\})$ if $\ell_\Gamma(\gamma_i) = \ell_\Gamma(\gamma_j)$. 
	\end{remark}
	\begin{proof}

		By  \cref{thm:new-characterization} it suffices to show that under the given assumptions, for any $\varepsilon > 0$ there is $\phi \in C_c(\G)$ with 
		$\phi_{\vert_{\Go}} = 1_{\Go}$
		such that $\phi$ is $\cK$-continuous and 
		\begin{align}\label{eq:BLkn-epsilon-approximation}
			\sup_{f\in E_{\cK, n}} \Vert f - m_{\phi}(f) \Vert_{\Cred(\G)} < \varepsilon.
		\end{align}
		Let therefore $\varepsilon > 0$ be given. Like in the proof of \cite[Theorem 2.6]{ChristensenIvanRD}, observe that if $1\leq t \leq \ell(\gamma)$, then for any $n > s$, in particular for $n \geq \lfloor s \rfloor +1$, we have
		\begin{align*}
			(1 + \ell(\gamma))^{2s} \leq 2^{2s} \ell(\gamma)^{2s} = 2^{2s} \ell(\gamma)^{2n} \ell(\gamma)^{2(s-n)} \leq 2^{2s} \ell(\gamma)^{2n} t^{2(s-n)}.
		\end{align*}
		Since $2s- 2n < 0$, we may choose a $t > 0$ such that $2^{2s} t^{2s -2n} < \frac{\varepsilon^2}{C^2}$. Fix such a $t$. 
		
		By assumption on the image of $\ell$ there is $\delta_t$ such that  $(t , t+\delta_t) \cap \ell(\G) = \emptyset$. 
		Thus $B_\ell(t)$ is a clopen subset of $\G$ by continuity of $\ell$. It follows that $\phi = 1_{B_\ell(t)} \in C_c(\G)$.  
		By \cref{prop:Cc-is-in-FS}, we have that $\phi \in \FS(\G)$, and so we get a completely bounded multiplier $m_\phi \colon \Cred(\G) \to \Cred(\G)$. 
		Then, by definition of $\phi$ and assumptions on $K_i$, $i \in I$, we either have $\phi_{\vert_{K_i}} = 1$ or $\phi_{\vert_{K_i}} = 0$ for every $i \in I$. As such, $m_\phi$ is $\cK$-continuous. Note in particular that $\phi_{\vert_{\Go}} = 1$. Thus it suffices to verify that \eqref{eq:BLkn-epsilon-approximation} holds for this $\phi$. 
		
		To this end fix an arbitrary $f\in E_{\cK, n}$. Using the rapid decay condition, we calculate  
		\begin{align*}
			&\Vert f - m_{\phi}(f) \Vert^2 \leq C^2 \Vert f - m_{\phi}(f) \Vert_{2,s}^2 \\
			&= C^2 \max \{  \sup_{u \in \Go} \sum_{\gamma \in \G_u} \vert f(\gamma) - \phi(\gamma)f(\gamma)\vert^2 (1+ \ell(\gamma))^{2s} ,  \sup_{u \in \Go} \sum_{\gamma \in \G_u} \vert f(\gamma^{-1}) - \phi(\gamma^{-1})  f(\gamma^{-1})\vert^2 (1+ \ell(\gamma))^{2s} \}  
		\end{align*}
		To continue the calculation, we fix $u \in \Go$ and show that independently of $u$ we have 
		\begin{align*}
			\sum_{\gamma \in \G_u} \vert f(\gamma) - \phi(\gamma)f(\gamma)\vert^2 (1+ \ell(\gamma))^{2s} < \varepsilon^2/C^2.
		\end{align*}
		The analogous estimate may be done for $\sum_{\gamma \in \G_u} \vert f(\gamma^{-1}) - \phi(\gamma^{-1})  f(\gamma^{-1})\vert^2 (1+ \ell(\gamma))^{2s}$. 
		
		Since $f \in E_{\cK, n}$ we in particular have
		\begin{align*}
			\sum_{\gamma \in \G_u} \vert f(\gamma)\vert^2 \ell(\gamma)^{2n} \leq \Vert \delta^n (f)\Vert_{\curlyL_{C(\Go)}(\E)}^2 \leq L^{\cK, n}(f)^2 \leq 1.
		\end{align*}
		By the construction of $\phi$ and since $n > s$, we may calculate 
		\begin{align*}
			\sum_{\gamma \in \G_u} \vert f(\gamma) &- \phi(\gamma)f(\gamma)\vert^2 (1+ \ell(\gamma))^{2s}\\ 
			&= \sum_{\gamma \in \G_u \cap B_\ell(t)} \vert f(\gamma) - \phi(\gamma)f(\gamma)\vert^2 (1+ \ell(\gamma))^{2s} + \sum_{\gamma \in \G_u\setminus B_\ell(t)} \vert f(\gamma) - \phi(\gamma)f(\gamma)\vert^2 (1+ \ell(\gamma))^{2s} \\
			&= 0 + \sum_{\gamma \in \G_u\setminus B_\ell(t)} \vert f(\gamma)\vert^2 (1 + \ell(\gamma))^{2s} \\
			&\leq 2^{2s}t^{2s-2n} \cdot \sum_{\gamma \in \G_u\setminus B_\ell(t)} \vert f(\gamma)\vert^2 \ell(\gamma)^{2n}  \\ 
			&<   \frac{\varepsilon^2}{C^2} \cdot \Vert \delta^n(f) \Vert_{\curlyL_{C(\Go)}(\E)}^2 \leq \frac{\varepsilon^2}{C^2}. 
		\end{align*}
		Since all estimates are independent of $u$, this finishes the proof 
		of the first part. 
		The second statement now follows by the first by noting that if $(\G,\ell)$ has polynomial growth bounded by a polynomial $p$ of degree $d$, then $(\G,\ell)$ has rapid decay using any $s > d +\frac{1}{2}$ by \cref{prop:degree-of-RD-from-PG}.
	\end{proof}

	\section{AF groupoids}\label{sec:AF-groupoids}	
	We begin this section by reminding the reader about AF groupoids, in particular those with compact unit spaces. 
	Suppose $\G$ is a second-countable ample \'etale groupoid, where $\Go$ is a 
	totally disconnected compact Hausdorff space. 
	We then say that $\G$ is an \emph{AF-groupoid} if there exists an increasing sequence $\G_1 \subseteq \G_2 \subseteq \ldots \subseteq \G$ consisting of clopen subgroupoids for which
	\begin{itemize}
		\item $\G_n$ is principal for every $n\in \N$;
		\item $\G_n^{(0)}  = \Go$ for every $n \in \N$;
		\item $\G_n$ 
		is compact for every $n \in \N$;
		\item $\G = \cup_{n=1}^\infty \G_n$.
	\end{itemize}
	Note that such a $\G$ must be principal. Moreover, we note that AF groupoids are in general not compactly generated. As shown in \parencite[Theorem 3.9]{GiordanoPutnamSkau2004}, any AF groupoid may be realized as the groupoid arising from a Bratteli diagram,
	and this is the manner in which we will realize them. Indeed, as the length function we will be concerned with arises from a Bratteli diagram and an in-depth understanding of the construction is necessary for the proof of \cref{thm:CQMS-from-AF-groupoids}, we provide the details on how an AF groupoid arises from a Bratteli diagram.
	
	A \emph{Bratteli diagram} $B = (V,E)$ is a directed graph whose vertex set $V$ and edge set $E$ may be written as countable disjoint unions of non-empty finite sets, that is,
	\begin{align*}
		V =  \coprod_{n=0}^\infty V_n \quad \text{ and } \quad E =  \coprod_{n = 1}^\infty E_n ,
	\end{align*}
	along with maps $i \colon E_n \to V_{n-1}$ and $t \colon E_n \to V_n$ 
	subject to the additional relations $i(E_n) = V_{n-1}$ and $t(E_n) \subseteq V_n$ for $n\geq 1$. The map $i$ is known as the \emph{source map}, and $t$ is known as the \emph{range map}. We will refer to the vertices in $V_n$ as being in \emph{level $n$} for simplicity. Denote by $S(B)$ the set of all sources in $B$, that is, the set of vertices $v \in V$ for which there are no edges with $v$ as target, that is, there is no $e \in E$ with $t(e) = v$. 
	While a Bratteli diagram could have infinitely many sources, 
	we will only be interested in groupoids with compact unit spaces, which will correspond to the Bratteli diagram only having finitely many sources. 
	
	To associate a groupoid to a Bratteli diagram $B = (V,E)$, we first construct its infinite path space. Given a source $v \in S(B) \cap V_n$, the set of infinite paths starting at $v$ is the set
	\begin{align*}
		X_v = \{ e_{n+1} e_{n+2} \ldots \mid e_i \in E_i, i(e_{n+1}) = v, \text{ and } i(e_{n+k+1}) = t(e_{n+k}) \text{ for all $k\geq 1$} \}.
	\end{align*}
	As a piece of suggestive notation, for $x \in X_v$ we will write $x = x_{n+1} x_{n+2} \ldots = x_{[n+1, \infty)}$ where $x_i \in E_i$ for all $i \geq n+1$. The infinite path space associated to $B$ is then, as a set, given by
	\begin{align*}
		X_B = \coprod_{v \in S(B)} X_v
	\end{align*}
	We endow $X_B$ with the topology which has a basis of compact open cylinder sets defined by the finite paths. Specifically, given a finite path $\mu$ with with $i(\mu) \in S(B) \cap V_n$ and length denoted by $\vert \mu \vert$, the cylinder set associated with $\mu$ is
	\begin{align*}
		Z(\mu) = \{ e_{n+1} e_{n+2} \ldots \in X_{i(\mu)} \mid e_{n+1} \ldots e_{n+ \vert \mu \vert} = \mu  \}.
	\end{align*}
	From the infinite path space, we may now construct a groupoid. 
	For every $N \geq 1$ we define
	\begin{align}\label{eq:def-inf-path-space}
		P_N = \{ (x,y) \in X_B^2 \mid i(x) \in V_m \cap S(B), i(y) \in V_n \cap S(B), m,n\leq N, x_k = y_k \forall k> N   \},
	\end{align}
	which we may view as the set of pairs of paths which eventually agree. Equipped with the relative topology, $P_N$ is a compact principal ample Hausdorff groupoid with $((x,y),(z,w)) \in P_N^{(2)}$ if and only if $y=z$, and so
	\begin{align}\label{eq:PN-gpd-operations}
		(x,y)(y,z) = (x,z) \quad \text{and} \quad (x,y)^{-1} = (y,x).
	\end{align}
	The unit space may be identified with
	\begin{align*}
		\coprod_{n=0}^N \coprod_{v \in S(B) \cap V_n} Z(v).
	\end{align*}
	We then define the groupoid $\G = \G_B$ associated with the Bratteli diagram $B=(V,E)$ as the increasing union
	\begin{align*}
		\G := \bigcup_{N\geq 1} P_N
	\end{align*}
	equipped with the inductive limit topology. The groupoid multiplication and inversion are given by the natural extensions of \eqref{eq:PN-gpd-operations}. Of importance to us is that the topology has a basis defined by pairs of finite paths $\mu$ and $\lambda$ with $i(\mu) \in S(B) \cap V_m$, $i(\lambda) \in S(B) \cap V_n$ for $m,n \in \N \cup \{0\}$, and with $t(\mu) = t(\lambda)$. The cylinder set is given by
	\begin{align}\label{eq:def-cylinder-set}
		Z(\mu, \lambda) = \{ (x,y) \in Z(\mu) \times Z(\lambda) \mid x_{[m + \vert \mu \vert + 1, \infty)} = y_{[n + \vert \lambda \vert + 1 , \infty)} \}.
	\end{align}
	The unit space of $\G$ may be identified with $X_B$, and it is compact if and only if $B$ has finitely many sources. 
	Setting 
	\begin{align}\label{def:cpct-subgroupoids-Gn}
		\G_n := P_n \cup \Go ,
	\end{align}
	the groupoid $\G$  is  realized as an increasing union of principal  clopen compact subgroupoids as follows
	\begin{align*}
		\G := \bigcup_{n=1}^\infty \G_n .
	\end{align*} 

We proceed to introduce a length function on $\G$ constructed from the Bratteli diagram $B = (V,E)$, which in turn will be used in the construction of compact quantum spaces below. For this we introduce the following notation. 
For $v, w \in V$ we will by $P(v,w)$ denote the set
\begin{align*}
	P(v,w) := \{ \mu  \mid \mu \text{ is a finite path with } i(\mu) = v \text{ and } t(\mu) = w   \}.
\end{align*}
If $W \subset V$ is a subset of the vertices, define 
\begin{align*}
	P(v,W) := \{ \mu  \mid \mu \text{ is a finite path with } i(\mu) = v \text{ and } t(\mu) \in W   \}.
\end{align*}
We may now define the  length function we will consider in this section. 

	\begin{definition}\label{def:AF-length-function} 
	Let $\G$ be an AF groupoid arising from a Bratteli diagram $B = (V,E)$ with $\vert S(B)\vert < \infty$. 
	For $(x,y) \in \G$, denote by 
	\begin{align*}
		k(x,y) = \min \{m \mid x_{[m+1,\infty)} = y_{[m+1,\infty)}\}.
	\end{align*}
	 Moreover, we define the function
	\begin{align}\label{eq:def-AF-length-function}
		\ell(x,y) = \begin{cases}
			0 & \text{ if  $x=y$} \\ 
			\sum_{v \in S(B)} \vert P(v, V_{k(x,y)}) \vert & \text{ if $x \neq y$}.
		\end{cases}
	\end{align}
\end{definition}

	We record the following basic observation about $k$ and $\ell$.

	\begin{lemma}\label{lemma:relating-ell-and-k}
		Let $\G$ be an AF groupoid arising from a Bratteli diagram $B = (V,E)$ with $\vert S(B)\vert < \infty$. 
		Let $k$ and $\ell$ be as in \cref{def:AF-length-function}. 
		Suppose $(x,y), (z,w) \in \G$ with $x\neq y$ and $z\neq w$. If $\ell(x,y) > \ell(z,w)$, then $k(x,y) > k(z,w)$. 
	\end{lemma}
	
	\begin{proof}
		Note that $x\neq y$ and $z\neq w$ implies that $(x,y), (z,w) \not\in \Go$, and hence
		\begin{align*}
			 \sum_{v \in S(B)} \vert P(v,V_{k(x,y)}) \vert = \ell(x,y) > \ell(w,z) =  \sum_{v \in S(B)} \vert P(v,V_{k(w,z)}) \vert,
		\end{align*} 
		from which we see that $k(x,y) > k(w,z)$. 
	\end{proof}

	To employ the machinery from \cref{sec:CQMS-from-groupoids}, it will also be important that the function $\ell$ is a continuous and proper length function, which we proceed to verify in the next two results. 
	
	\begin{proposition}\label{prop:ellB-is-cont-length-function}
		Let $\G$ be an AF groupoid arising from a Bratteli diagram $B = (V,E)$ with $\vert S(B)\vert < \infty$.  
		The function $\ell$ from \eqref{eq:def-AF-length-function} 
		is a continuous length function on 
		$\G$.  
	\end{proposition}
	
	\begin{proof}
		We first verify that $\ell$ is a length function. 
		First, note that $\ell(x,x) = 0$ by definition. Moreover, if $x\neq y$, then $k(x,y)\geq 1$ and so $P(v, V_{k(x,y)}) \neq \emptyset$ for at least one $v \in S(B)$. We conclude that $\ell(x,y) = 0$ if and only if $x = y$. 
		Since $(x,y)^{-1} = (y,x)$ in $\G$, we see that $\ell((x,y)^{-1}) =  \ell(x,y)$, since $k(x,y) = k(y,x)$. 
		Now, let $(x,y), (y,z), (x,z) \in \G$. We wish to show that
		\begin{align*}
			\ell(x,z) \leq \ell(x,y) + \ell(y,z).
		\end{align*}
		Suppose for the sake of contradiction that $\ell(x,z) > \max \{ \ell(x,y), \ell(y,z)\}$. By \cref{lemma:relating-ell-and-k} we then have $k(x,z) > \max \{ k(x,y), k(y,z) \}$. But then $x_{[l_1, \infty)} = y_{[l_1, \infty)}$ for some $l_1 < k(x,z)$, and $y_{[l_2, \infty)} = z_{[l_2, \infty)}$ for some $l_2 < k(x,z)$.   We conclude that $x_{[l, \infty)} = z_{[l, \infty)}$ for $l < k(x,z)$ contradicting the minimality of $k(x,z)$. We deduce a stronger version of the triangle inequality, namely $\ell(x,z) \leq \max \{ \ell(x,y) , \ell(y,z) \}$.

		To see that $\ell$ is continuous, suppose $(x^{(j)}, y^{(j)}) \to (x,y)$ in $\G$. There are finite paths $\mu, \lambda$ with $t(\mu) = t(\lambda)$ such that $(x,y) \in Z(\mu, \lambda)$, cf. \eqref{eq:def-cylinder-set}. We then have $x_{[m +\vert \mu  \vert + 1 , \infty)} = y_{[n + \vert \lambda \vert + 1 , \infty)}$, and suppose that $m + \vert \mu \vert$ is minimal with this property, that is, $x$ and $y$ do not start agreeing at an earlier level. Since $Z(\mu, \lambda)$ is open, there is $N$ such that for all $j \geq N$, $(x^{(j)}, y^{(j)}) \in Z(\mu, \lambda)$. We claim that $N$ can also be chosen so that $x^{(j)}$ and $y^{(j)}$ do not agree at a level earlier than $m + \vert \mu \vert = n + \vert \lambda \vert$. Indeed, if no such $N$ exists, that is, if for arbitrarily large $j$, there is $l_j < m + \vert \mu \vert +1$ such that $x^{(j)}_{[l_j, \infty)} =  y^{(j)}_{[l_j, \infty)}$, we would have
		\begin{align*}
			s(x^{(j)}, y^{(j)}) \not \to s(x,y) = y \quad \text{or} \quad r(x^{(j)}, y^{(j)}) \not \to r(x,y) = x,
		\end{align*}
		thus implying that $s$ or $r$ is not continuous, a contradiction. Thus $x^{(j)}$ and $y^{(j)}$ eventually agree from level $m + \vert \mu \vert $, and $\ell (x^{(j)}, y^{(j)}) = \ell (x,y)$. Thus $\ell$ is continuous. 
	\end{proof}

	\begin{proposition}\label{prop:ell-preimage-in-subgpd}
		Let $\G$ be an AF groupoid arising from a Bratteli diagram $B = (V,E)$ with $\vert S(B)\vert < \infty$. 
		Let $\ell$ be the length function from 
		\eqref{eq:def-AF-length-function}. 
		For any $M \geq 0$ there is $N(M)$ such that $B_\ell(M) \subseteq \G_{N(M)}$. In particular, $\ell$ is a proper length function. 
	\end{proposition} 
	\begin{proof}
		Note first that if $\ell(\G) \subseteq [0, R]$ for some  $R< \infty$, then the Bratteli diagram must stabilize, that is, for $N$ large enough the infinite path space $P_N$ from \eqref{eq:def-inf-path-space} is such that $P_{N+k} = P_N$ for all $k \geq 0$.  
		Then $\G$ is itself a compact groupoid, and $B_\ell(M) \subseteq \G$ is compact for any $M \geq 0$ as $\ell$ is continuous and  $B_\ell(M)$ is therefore a closed subset of a compact groupoid.  
		Thus $\ell$ is proper.
		
		Suppose therefore that $\ell$ is an unbounded function, so that  
		the Bratteli diagram does not stabilize.
		Let $M \geq 0$. Since $\ell$ counts the number of paths from the sources up to a level, there must be $N(M)$ such that $B_\ell(M) \subseteq \G_{N(M)}$. As before, we deduce that $B_\ell(M)$ is a closed subset of a compact groupoid, and therefore compact itself. The statement follows.

	\end{proof}
	
	\begin{remark}
		Note that different Bratteli diagrams could be associated with the same AF groupoid, therefore leading to different length functions. 
	\end{remark}

	\begin{remark}\label{remark:Gn-is-ball} 
		We will later want to relate the subgroupoids $\G_n$ to the balls $B_\ell(R)$, $R \geq 0$. 
		It could happen that the number of paths from the sources to $V_n$ is the same as the number of paths from the sources to $V_{n+1}$, while $\G_{n+1} \setminus \G_n \neq \emptyset$. This happens for example if there is only one edge from each vertex in $V_n$ to vertices in $V_{n+1}$, but one of the vertices in $V_{n+1}$ receives (at least) two edges. It is therefore not necessarily the case that every $\G_n$ is a ball for $\ell$.  However, for any $n$ there is $m \geq n$ such that $\G_{m}$ equals $B_\ell(R_m)$ for some $R_m \geq 0$.  If $\G$ is compact, this holds as $\ell$ is then bounded and $\G = \G_m$ for sufficiently large $m$.  $\G$ then equals the $\ell$-ball of any sufficiently large radius. If $\G$ is not compact, then $\ell$ must be unbounded since it is proper by \cref{prop:ell-preimage-in-subgpd}. As $\ell$ counts the number of paths up to a level, we see that such an $m$ and accompanying $R_m$ must exist. 
	\end{remark}
	
	For the sake of completeness, we verify that the metric groupoid $(\G, \ell)$ has at most linear growth, and therefore has rapid decay. 

	\begin{lemma}\label{lemma:ellB-gives-linear-growth}
		Let $\G$ be an AF groupoid  arising from a Bratteli diagram $B = (V,E)$ with $\vert S(B)\vert < \infty$.
		Then the groupoid $\G$ has at most linear growth with respect to the length function $\ell$ from \eqref{eq:def-AF-length-function}. Consequently, $(\G, \ell)$ has rapid decay.  
	\end{lemma}
	
	\begin{proof}
		If $\G$ is compact we are done since $\ell$ is continuous and will therefore  
		be bounded. The polynomial in \eqref{eq:poly-growth-criterion} may then be chosen to be a constant.  
		Hence we may assume that $\G$ is not compact and therefore that $\ell$ 
		is unbounded by \cref{prop:ell-preimage-in-subgpd}.
		Let $y \in \G^{(0)}$ be arbitrary. We wish to show that $ t\mapsto \vert \{ \gamma \in \G_y \mid \ell(\gamma)\leq t    \} \vert $ is linearly bounded independently of $y$. Let $k \geq 0$ be the largest integer for which 
		\begin{align*}
			\sum_{v \in S(B)} \vert P(v, V_k) \vert \leq M.
		\end{align*}
		If $(x,y) \in \G$ is such that $\ell(x,y) \leq M$, then $x_m = y_m$ for all $m > k$, but for $m\leq k$ we could have $y_m \neq x_m$. The number of possible such $x$ starting in vertex $v \in S(B)$ is then $\vert P(v, t(y_k))\vert$, from which we get
		\begin{align*}
			\vert \{ (x,y) \in \G_y \mid \ell(x,y)\leq M    \} \vert = \sum_{v \in S(B)} \vert P(v , t(y_k)) \vert \leq \sum_{v \in S(B)} \vert P(v , V_k) \vert \leq M
		\end{align*}
		from which it follows that $\G$ has linear growth with respect to $\ell$. It then follows that $(\G, \ell)$ has rapid decay by \cref{prop:degree-of-RD-from-PG}. 
	\end{proof}
	
	Recall that an AF groupoid $\G$ with compact unit space, as well as all the subgroupoids $\G_n$, $n \in \N$ are principal. Moreover, $\Go$ is a 
	totally disconnected compact Hausdorff space. 
	Pick any metric $d$ on $\Go$ inducing the topology. 
	By \cref{lemma:principal-metric-stratification} we then have a metric stratification of $\G$ with respect to $d$ given by
	\begin{align}\label{eq:metric-stratification-AF}
		\cK = (K_i)_{i \in \N \cup \{0\}} = (\ell^{-1}(\{i\}))_{i\in \N \cup \{0\}}.
	\end{align}
	Note that many of the $K_i$ will be empty as in general not every natural number will be in the range of $\ell$. 
	
	We record the following result, guaranteeing that natural sub-systems determined by the clopen compact subgroupoids $\G_n$, $n \in \N$, yield compact quantum metric spaces.

	\begin{proposition}\label{prop:CQMS-from-compact-AF-subgroupoid}
		Let $\G$ be an AF groupoid arising from a Bratteli diagram $B = (V,E)$ with $\vert S(B)\vert < \infty$. Furthermore, let $\ell$ be given by \eqref{eq:def-AF-length-function}, and let $d$ be a metric on $\Go$ inducing the topology. 
		We set $K_i = \ell^{-1}(\{i\})$, $i \in \N \cup \{0\}$ to be the a metric stratification of $\G$ as in \eqref{eq:metric-stratification-AF}.
		Denote the resulting seminorm using $n$ iterated commutators by $ L^{\cK, n}$ as in \cref{def:total-seminorm}. 
		For each $m \in \N$, let $\G_m \subseteq \G$ be the compact clopen subgroupoid given by \eqref{def:cpct-subgroupoids-Gn}. 
		Equip $\G_m$ with the restricted metric stratification given by $\cK_{m} := \cK \cap \G_m$. 
		
		Then $(\Lip_c^{\cK_m}(\G_m), (L^{\cK, n})_{\vert_{\Lip_c^{\cK_m}(\G_m)}})$ is a compact quantum metric space for every $n\geq 1$.
	\end{proposition}
	
	\begin{proof}
		With $\cK_{m} := \cK \cap \G_m$, we see that $\Lip_c^{\cK_m}(\G_m)$ is a sub-operator system of 
		$\Lip_c^{\cK}(\G)$
		through extending functions by zero outside of $\G_m$, since $\G_m \subseteq \G$ is a clopen subgroupoid. Since $\ell$ is continuous, it is bounded on compact sets. So there is $M_m \geq 0$ such that $\supp (f) \subseteq B_{\ell}(M_m) \subseteq \G$ for all $f \in \Lip_c^{\cK_m}(\G_m)$. The statement then follows by \cref{lemma:compact-subsets-yield-CQMS}. 
	\end{proof}
	
	\begin{remark}\label{remark:ghost-paths}
		Before stating and proving \cref{thm:CQMS-from-AF-groupoids} below, we make an observation which will simplify notation and streamline the proof. In particular, it will allow us to pretend that the Bratteli diagram only has a single source. Suppose $B = (V,E)$ is a Bratteli diagram. Let $K = \max\{ k \mid S(B) \cap V_k \} \neq \emptyset$. We augment $B$ in the following way
		\begin{itemize}
			\item Add a single vertex $\overline{v}$ at level $-1$.
			\item For each $v \in S(B) \cap V_k$, $0 \leq k\leq K$ define $u_v^k := v$. Moreover, for each $v \in S(B) \cap V_k$, $k \geq 1$, add a vertex $u_v^{i}$ to $V_i$ for each $0 \leq i \leq k-1$. 
			\item For each $v \in S(B)$, add a single edge between $\overline{v}$ and $u_v^0$.  
			Then, for all $v \in S(B) \cap V_k$, $k\geq 1$, add a single edge from $u_v^{i}$ to $u_v^{i+1}$ for each $i=0, \ldots, k-1$.
		\end{itemize}
		The resulting diagram should have a unique path from $\overline{v}$ to each $v \in S(B)$ going through the vertices $u_v^i$. 
		In fact, the resulting diagram is a Bratteli diagram (starting at level $-1$ rather than $0$) for the same AF algebra.  
		If we venture far enough out in the diagram we may relate our length function to the number of paths in the augmented diagram: 
		For all $k \geq K +1$ we observe that
		\begin{align*}
			\sum_{v \in S(B) } \vert P(v, V_k) \vert = \vert P(\overline{v}, V_k) \vert .
		\end{align*}
		Indeed,  since there are no paths between the different $v \in S(B)$ this just follows from 
		\begin{align*}
			\vert P(\overline{v}, V_k) \vert = \sum_{v \in S(B)}\vert P(\overline{v}, v)\vert \cdot \vert P(v, V_k)\vert = \sum_{v \in S(B)} 1 \cdot \vert P(v, V_k)\vert = \sum_{v \in S(B)}\vert P(v, V_k)\vert
		\end{align*}
	for all $k \geq K+1$.  
	\end{remark}
	
		We may now prove the main theorem of this section, showing that AF groupoids  induce compact quantum metric space structures through data naturally associated with it, namely a Bratteli diagram and a metric on the unit space. Note that despite the linear growth result from \cref{lemma:ellB-gives-linear-growth}, the statement does not follow immediately from \cref{prop:Christensen-Ivan-for-groupoids}, 
		as this only yields that $n\geq 2$ commutators would suffice, see \cref{prop:degree-of-RD-from-PG}.   

	\begin{theorem}\label{thm:CQMS-from-AF-groupoids}
		Let $\G$ be an AF groupoid arising from a Bratteli diagram $B = (V,E)$ with $\vert S(B)\vert < \infty$. Furthermore, let $\ell$ be given by \eqref{eq:def-AF-length-function}, and let $d$ be a metric on $\Go$ inducing the topology. 
		Equip $\G$ with the metric stratification $\cK = (K_i)_{i \in \N \cup \{0\}}$ given by $K_i = \ell^{-1}(\{i\})$ from \eqref{eq:metric-stratification-AF}, and denote the resulting seminorm using $n$ commutators by $L^{\cK, n}$ as in \cref{def:total-seminorm}.
		\begin{enumerate} 
			\item $(\Lip_c^{\cK}(\G), L^{\cK, n})$ 
			is a compact quantum metric space for all $n \geq 1$.
			\item If we equip the compact subgroupoids $\G_m$ with the restricted metric stratifications as in \cref{prop:CQMS-from-compact-AF-subgroupoid}, the sequence of compact quantum metric spaces $(\Lip_c^{\cK_m}(\G_m), (L^{\cK, n})_{\vert_{\Lip_c^{\cK_m}(\G_m)}})$ converges to  
			$(\Lip_c^{\cK}(\G), L^{\cK, n})$
			in quantum Gromov--Hausdorff distance. In fact, 
			\begin{align}\label{eq:rate-of-qGH-convergence}
				\mathrm{dist}_Q (( \Lip_c^{\cK}(\G), L^{\cK, n}), (\Lip_c^{\cK_m}(\G_m)&,(L^{\cK, n})_{\vert_{\Lip_c^{\cK_m}(\G_m)}})^2 \notag \\
				&\leq \sup_{y \in \Go} \sum_{(x,y) \in (\G_y \setminus (\G_m)_y )} \ell(x,y)^{-2} \to 0
			\end{align}
			as $m \to \infty$. 
		\end{enumerate}
	\end{theorem}

	\begin{proof}
		Note first that if $\G$ is compact, then the result follows from \cref{prop:CQMS-from-compact-AF-subgroupoid}. We will therefore assume throughout the proof that $\G$ is not compact, and thus that the length function $\ell$ is unbounded. 
		
		We wish to employ the norm approximation condition from \cref{thm:new-characterization}. To prove the present theorem, we will use Fourier multipliers coming from restriction down to the subgroupoids $\G_m$, and instead of the operator norm expression from \eqref{eq:norm-approximation-yields-CQMS}, we will use the corresponding $I$-norm expression. This will suffice since the $I$-norm dominates the reduced groupoid $C^*$-algebra norm. 
		Note also that it suffices to prove the statement for $n=1$ commutators, as the cases $n \geq 2$ follows from this. 
		Let
		$E_{\cK, 1}$ 
		be
		as in \eqref{eq:def-E-set}, 
		and let 
		$\varepsilon > 0$ be arbitrary.
		By \cref{thm:new-characterization} it suffices to show there is $\phi \in C_c(\G)$ with 
		$\phi_{\vert_{\Go}} = 1_{\Go}$
		which is $\cK$-continuous and for which 
		\begin{align*}
			\sup_{f\in E_{\cK, 1}} \Vert f - m_{\phi}(f) \Vert_{\Cred(\G)} < \varepsilon.
		\end{align*}
		We set $\phi = 1_{\G_m}$ for some $m$ which will be determined later.  For now however, note that we will choose $m > K:= \max \{ k \mid S(B) \cap V_k \neq \emptyset \}$. 
		We augment $B$ as in \cref{remark:ghost-paths} and see that 
		\begin{align*}
			\vert P(\overline{v}, V_k) \vert = \sum_{v \in S(B)}\vert P(v, V_k)\vert
		\end{align*}
		for all $k \geq K+1$, where $\overline{v}$ is the unique vertex in level $-1$ in the augmented diagram.

		Note that the Fourier multiplier $m_\phi$ associated with $\phi$ coincides with the conditional expectation $\Cred (\G) \to \Cred(\G_m) \subseteq \Cred(\G)$ from \cite[Theorem 6.2]{HirschbergWu2021}, which is a unital completely positive map. Therefore $\phi$ is a positive definite function by \cref{prop:pos-def-iff-cp-map}. We have $\phi \in C_c(\G) \subseteq \FS(\G)$ by \cref{prop:Cc-is-in-FS}. It follows by \cref{prop:multipliers-lift-to-crossed-product} that we get a completely positive multiplier ${T^\phi} \colon \Cred (\beta \G \rtimes \G) \to \Cred(\beta \G \rtimes \G)$. 
		
		Using \cref{remark:Gn-is-ball} we will choose $m \geq K+1$ so that $\G_m$ is itself a ball of radius $\vert P(\overline{v}, V_m)\vert$ in $\G$, and as such  the multiplier $m_\phi$ is $\cK$-continuous with coefficient $1$. Below we will furthermore specify a value $M > 0$. We note that $m$ will moreover be chosen in relation to $M$ such that $B_\ell(M) \subseteq \G_m$, which is possible for some $m$ large enough by \cref{prop:ell-preimage-in-subgpd}.

		Let $f \in E_{\cK, 1}$ be arbitrary. We calculate
		
		\begin{align}\label{eq:AF-multiplier-approximation}
			\Vert f - m_\phi (f) \Vert_{\Cred(\G)} &\leq \Vert f - 1_{\G_m} \cdot f \Vert_I \leq \Vert f - f_{\vert_{B_\ell(M)}} \Vert_I \\
			&= \max \{ \sup_{y \in \Go} \sum_{\substack{(x,y)\in \G_y \\ \ell(x,y) > M}} \vert f(x,y) \vert , \sup_{y \in \Go} \sum_{\substack{(x,y)\in \G_y \\ \ell(x,y) > M}} \vert f^*(x,y) \vert    \} \notag
		\end{align}
		Now note that for every $y \in \Go$ and $M > 0$
		\begin{align}\label{eq:AF-CS-estimate}
			\sum_{\substack{(x,y)\in \G_y \\ \ell(x,y) > M}} \vert f(x,y) \vert &=  \sum_{\substack{(x,y)\in \G_y \\ \ell(x,y) > M}} \ell(x,y)^{-1} \ell(x,y)\vert f(x,y) \vert  \\
			&\leq \big(  \sum_{\substack{(x,y)\in \G_y \\ \ell(x,y) > M}} \ell(x,y)^{-2} \big)^{1/2} \cdot \big( \sum_{\substack{(x,y)\in \G_y \\ \ell(x,y) > M}} \vert f(x,y) \vert^2 \ell(x,y)^2  \big)^{1/2} \notag \\
			&\leq \big(  \sum_{\substack{(x,y)\in \G_y \\ \ell(x,y) > M}} \ell(x,y)^{-2} \big)^{1/2} \cdot L^1_\ell(f) \notag 
			\leq  \big(  \sum_{\substack{(x,y)\in \G_y \\ \ell(x,y) > M}} \ell(x,y)^{-2} \big)^{1/2} \notag
		\end{align}
		since $f \in E_{\cK,1}$. The same calculation holds for $f^*$ since $f^* \in E_{\cK,1}$ also. Thus it suffices to show that there exists $M > 0$ such that
		\begin{align*}
			\sup_{y \in \Go} \sum_{\substack{(x,y)\in \G_y \\ \ell(x,y) > M}} \ell(x,y)^{-2} < \varepsilon.
		\end{align*}
		Now fix $y \in \Go$ and note that for any $d \in \N \cup \{0\}$ which is in the image of $\ell$, that is $d \in (\N \cup \{0\}) \cap \ell(\G)$, there is a largest 
		number $l(d)$ for which $\sum_{v \in S(B)} \vert P(v, V_{l(d)})\vert = d$. For $l(d) \geq K+1$, we get
		\begin{align*}
			d = \sum_{v \in S(B)} \vert P(v, V_{l(d)})\vert = \vert P(\overline{v} , V_{l(d)}) \vert.
		\end{align*}
		Choosing $M$ so large that $l(M+1) > K$, we may then calculate
		\begin{align*}
			\sum_{\substack{(x,y)\in \G_y \\ \ell(x,y) > M}} \ell(x,y)^{-2} &= \sum_{k = M+1}^\infty \sum_{\substack{(x,y)\in \G_y \\ \ell(x,y) = k}} \frac{1}{\ell(x,y)^2} \\ 
			&= \sum_{k = M+1}^\infty \frac{\vert \{ (x,y) \in \G_y \mid \ell(x,y) = k  \}\vert}{k^2} \\
			&= \sum_{h = l(M+1)}^{\infty} \frac{\vert \{ (x,y)\in \G_y \mid \ell(x,y) =\vert P(\overline{v} , V_{h}) \vert \}\vert}{\vert P(\overline{v} , V_{h}) \vert^2}.
		\end{align*}
		We are therefore interested in the number of paths from $\overline{v}$ to $t(y_h) \in V_h$. 
		We may express the number of paths as follows
		\begin{align*}
			\vert P(\overline{v}, t(y_{h})) \vert = \vert P(\overline{v}, t(y_{h-1}))\vert \cdot ( \vert P(t(y_{h-1}), t(y_{h})) \vert -1 )  ) +A_h
		\end{align*}
		where
		\begin{itemize}
			\item the first summand counts the number of paths up to the previous vertex that $y$ went through, that is $t(y_{h-1})$, and multiplies it with the number of edges between $t(y_{h-1})$ and $t(y_{h})$, which are not the edge $y_{h}$ itself,
			\item the term $A_h$ is the number of paths coming into $t(y_h)$ from vertices in level $h-1$ which are not $t(y_{h-1})$. 
		\end{itemize}
		Therefore
		\begin{align*}
			\sum_{h = l(M+1)}^{\infty}& \frac{\vert \{ (x,y)\in \G_y \mid \ell(x,y) = \vert P(\overline{v}, V_h) \vert \}\vert}{\vert P(\overline{v}, V_h) \vert^2} \\
			&= \sum_{h = l(M+1)}^{\infty} \frac{ \vert P(\overline{v}, t(y_{h-1}))\vert ( \vert P(t(y_{h-1}), t(y_{h})) \vert -1   ) + A_h}{\vert P(\overline{v} , V_h) \vert^2}
		\end{align*}
		We split this into two separate sums, one for each of the terms in the numerator. That is, we independently show that 
		\begin{align*}
			I_1 := \sum_{h = l(M+1)}^{\infty} \frac{ \vert P(\overline{v}, t(y_{h-1}))\vert ( \vert P(t(y_{h-1}), t(y_{h})) \vert -1 )   }{\vert P(\overline{v} , V_h) \vert^2} < \varepsilon/2,
		\end{align*}
		and
		\begin{align*}
			I_2 := \sum_{h = l(M+1)}^{\infty} \frac{  A_h}{\vert P(\overline{v} , V_h) \vert^2} < \varepsilon/2,
		\end{align*}
		by choosing $M$ large enough. Noting that $\vert P(\overline{v} , V_h) \vert \geq \vert P(\overline{v}, t(y_{h-1}))\vert ( \vert P(t(y_{h-1}), t(y_{h}) \vert -1 )  )$, 
		we first calculate
		\begin{align*}
			I_1 &\leq \sum_{h = l(M+1)}^{\infty} \frac{ \vert P(\overline{v}, t(y_{h-1}))\vert ( \vert P(t(y_{h-1}), t(y_{h})) \vert -1 )   }{\vert P(\overline{v} , V_h) \vert^2} \\
			&\leq \sum_{\substack{h \geq l(M+1) \\ \vert P(t(y_{h-1}), t(y_h))\vert \geq 2 }} \frac{ 1}{\vert P(\overline{v} , V_h) \vert }.
		\end{align*}
		The sum is only over levels where we get new contributions, so we reindex such that each term is non-zero. We end up with a new sequence $(B_p)_{p=1}^{\infty}$, 
		where $p$ corresponds to the pth instance in the above sum we have  $\vert P(t(y_{h-1}), t(y_h))\vert \geq 2 $. We then see that
		\begin{align*}
			&B_1 \leq \frac{ 1}{\vert P(\overline{v} , V_{l(M+1)}) \vert } \\
			&B_p \leq \frac{ 1}{\vert P(\overline{v} , V_{l(M+1)}) \vert + 2^{p-1} } \quad \text{for $p \geq 2$}
		\end{align*}
		since there are at least $\vert P(\overline{v} , V_{l(M+1)}) \vert + 2^{p-1}$ paths from $\overline{v}$ to the level after $l(M+1)$ where we have received $p-1$ contributions from the condition $\vert P(t(y_{h-1}), t(y_h))\vert \geq 2$. We then calculate
		\begin{align*}
			I_1 &\leq \sum_{p=1}^\infty B_p \leq \frac{ 1}{\vert P(\overline{v} , V_{l(M+1)}) \vert } + \sum_{p=2}^\infty \frac{ 1}{\vert P(\overline{v} , V_{l(M+1)}) \vert + 2^{p-1} } \\
			&\leq \frac{ 1}{\vert P(\overline{v} , V_{l(M+1)}) \vert } + \sum_{p=2}^\infty \frac{ 1}{\vert P(\overline{v} , V_{l(M+1)}) \vert^{1/2} } \frac{1}{(2^{p-1})^{1/2}}\\
			&= \frac{1}{\vert P(\overline{v} , V_{l(M+1)}) \vert} + \frac{ 1}{\vert P(\overline{v} , V_{l(M+1)}) \vert^{1/2} } \cdot \sum_{p=1}^\infty  \frac{1}{(2^{p})^{1/2}} \\
			&= \frac{1}{\vert P(\overline{v} , V_{l(M+1)}) \vert} + \frac{ 1}{\vert P(\overline{v} , V_{l(M+1)}) \vert^{1/2} } \cdot (1 + \sqrt{2}).
		\end{align*}
		We 
		can 
		then
		definitely bound 
		$I_1$
		by $ \varepsilon/2$ by choosing ${l(M+1)}$ large enough, which is achieved by choosing $M$ large enough. The bound only depends on the value $\vert P(\overline{v} , V_{l(M+1)}) \vert$, so our estimate can be done uniformly in $y \in \Go$.

		For the second sum, we get a bound by pretending that the only increases to $\vert P(\overline{v}, V_h)\vert$ come from contributions $A_h$. In particular, we get a lower bound $\vert P(\overline{v} , V_{h+1}) \vert \geq \vert P(\overline{v}, V_h) \vert + A_h$. Repeatedly applying this observation we get
		\begin{align*}
			I_2 = \sum_{h = l(M+1)}^{\infty} \frac{  A_h}{\vert P(\overline{v} , V_h) \vert^2} 
			\leq \sum_{h= l(M+1)}^\infty \frac{A_h}{\left( \vert P(\overline{v} , V_{l(M+1)} \vert  + \sum_{j=l(M+1)}^{h} A_j \right)^2} .
		\end{align*}
		We have the following bound
		\begin{align*}
			\frac{A_h}{\left( \vert P(\overline{v} , V_{l(M+1)} \vert + \sum_{j=l(M+1)}^{h} A_j \right)^2} \leq \frac{1}{\vert P(\overline{v} , V_{l(M+1)} \vert^{1/2}} \cdot \frac{A_h}{(\sum_{j= l(M+1)}^{h} A_j)^{3/2}},
		\end{align*}
		from which we obtain
		\begin{align*}
			I_2 \leq \frac{1}{\vert P(\overline{v} , V_{l(M+1)}) \vert^{1/2}} \cdot \sum_{h = l(M+1)}^{\infty} \frac{A_h}{(\sum_{j= l(M+1)}^{h} A_j)^{3/2}}	.		
		\end{align*}
		It would therefore suffice to show that the latter sum, which we will call $I_2'$, is finite. 
		Set $B_t = A_{t -1 + l(M+1)}$ for all $t \geq 1$. Further, we may reindex such that all $B_t \geq 1$, as otherwise the summand will be zero. We calculate
		\begin{align*}
			I_2' :=\sum_{h = l(M+1)}^{\infty} \frac{A_h}{(\sum_{j= l(M+1)}^{h} A_j)^{3/2}} = \sum_{t=1}^{\infty} \frac{B_t}{(\sum_{s=1}^{t} B_s)^{3/2}} 
			= \sum_{t=0}^\infty \sum_{r=2^t}^{2^{t+1}-1} \frac{B_r}{( \sum_{s=1}^{r} B_s)^{3/2}} .
		\end{align*}
		Note now that for each interval $[2^t , 2^{t+1}-1]$ the numerator $(\sum_{s=1}^{r} B_s)^{3/2} \geq (2^t)^{3/2}$. We may therefore bound the sum as follows
		\begin{align*}
			I_2' \leq \sum_{t=1}^\infty \frac{2^{t+1} - 1 - 2^t}{(2^t)^{3/2}} \leq \sum_{t=1}^\infty \frac{2^t}{2^t \cdot 2^{t/2}} = \sum_{t=1}^{\infty} \frac{1}{2^{t/2}} = 1 + \sqrt{2} < \infty.
		\end{align*}
		Since $\vert P(\overline{v} , V_{l(M+1)} \vert^{-1/2}$ can be made arbitrarily small by choosing $M$ large enough, 
		we conclude that we can bound $I_2$ by  $\varepsilon/2$. 
		As above,  
		the bound only depends on the value $\vert P(\overline{v} , V_{l(M+1)}) \vert$, so our estimate can be done uniformly in $y \in \Go$.
		
		It follows that there is $M>0$ such that 
		\begin{align*}
			\sup_{y \in \Go} \sum_{\substack{(x,y)\in \G_y \\ \ell(x,y) > M}} \ell(x,y)^{-2} < \varepsilon
		\end{align*}
		and thus that 
		\begin{align*}
			\Vert f - m_{\phi}(f) \Vert \leq \Vert f - f_{\vert_{B_\ell(M)}} \Vert < \varepsilon.
		\end{align*}
		By \eqref{eq:AF-multiplier-approximation} and \cref{thm:new-characterization} we deduce that 
		$(\Lip_c^{\cK}(\G), L^{\cK, 1})$ 
		is a compact quantum metric space. As noted, the same conclusion holds for 
		$(\Lip_c^{\cK}(\G), L^{\cK, n})$ 
		for any $n \geq 1$. 
		
		We proceed to show statement (2). Following \cref{prop:KK-subsystem-qGH-convergence} we will, given any $\varepsilon > 0$, find a unital positive map 
		$\Phi_m \colon  \Lip_c^{\cK}(\G) \to \Lip_c^{\cK_m} (\G_m)$
		for which
		\begin{enumerate}
			\item[(a)] $\Vert f - \Phi_m (f) \Vert_{\Cred(\G)} \leq \varepsilon L^{\cK, n}(f)$ for all 
			$f \in \Lip_c^{\cK}(\G)$, 
			and
			\item[(b)] 
			$(L^{\cK, n})_{\vert_{\Lip_c^{\cK_m}(\G_m)}}(\Phi_m (f)) \leq L^{\cK, n}(f)$ 
			for all 
			$f \in \Lip_c^{\cK}(\G)$.
		\end{enumerate} 
		We choose $\Phi_m = m_{1_{\G_{m}}}$. Since $\cK_m = \cK \cap \G_m$, we see that $\Phi_m(f) \in \Lip_c^{\cK_m}(\G_m)$ for every $f\in \Lip_c^{\cK}(\G)$.  
			Moreover, as noted earlier, 
		$m_{1_{\G_{m}}}$ extends to a unital completely positive map $T^{1_{\G_{m}}}$ on $\Cred(\G \ltimes \beta\G  )$. From this it follows that $\Vert T^{1_{\G_{m}}} \Vert_{\mathrm{cb}} = \Vert m_{1_{\G_{m}}}\Vert_{\mathrm{cb}} = 1$ since the maps are unital and completely positive. Combining these observations, we see from \cref{lemma:multipliers-extend-to-Roe} that $\Phi_m$ is a slip-norm contraction, so (b) is satisfied.
		
		Furthermore, we have for any $\varepsilon > 0$ verified earlier in the proof that we may choose $m$ large enough so that
		\begin{align*}
			\Vert f - \Phi_m (f) \Vert_{\Cred(\G)} \leq \Vert f - \Phi_m (f) \Vert_I < \varepsilon \cdot L^{\cK, n}(f)
		\end{align*}
		for all 
		$f \in  \Lip_c^{\cK}(\G)$, 
		though earlier in the proof we set $n=1$ and used $f \in E_{\cK,1}$. Thus (a) is also satisfied. 
		
		From \cref{prop:KK-subsystem-qGH-convergence} it follows that for any $\varepsilon > 0$ there is $N$ such that for all $m \geq N$ we have
		\begin{align*}
			\mathrm{dist}_Q (( \Lip_c^{\cK}(\G), L^{\cK, n}), (\Lip_c^{\cK_m}(\G_m),(L^{\cK, n})_{\vert_{\Lip_c^{\cK_m}(\G_m)}})) <  \varepsilon.
		\end{align*}
		The very last inequality of \eqref{eq:rate-of-qGH-convergence} follows by redoing the calculations \eqref{eq:AF-multiplier-approximation} and \eqref{eq:AF-CS-estimate} without using the rough estimate $\Vert f - 1_{\G_m} \cdot f \Vert_I \leq \Vert f - f_{\vert_{B_\ell(M)}} \Vert_I$ like we did earlier. This finishes the proof.

	\end{proof}
	
	\begin{example}[The CAR algebra]
		The CAR algebra, see for example \cite[Example III.5.4]{DavidsonBook1996}, is a well-studied $C^*$-algebra with a groupoid model which admits a particularly simple Bratteli diagram and resulting length function. Note that one approach to quantum metric geometry of the CAR algebra has previously been noted by Aguilar in for example \cite[Example 2.4]{AguilarIndLim2021}. We provide a groupoid alternative to this using \cref{thm:CQMS-from-AF-groupoids}. Indeed, a Bratteli diagram $B=(V,E)$ for the CAR algebra is given by one vertex at each level, with two edges connecting level $k$ to $k+1$ for $k \in \N\cup\{0\}$ as follows
		
		
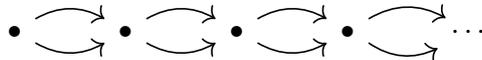
\vspace{4mm}
		\begin{center}
			\begin{tikzcd}
				\bullet \arrow[r, bend left] \arrow[r, bend right] & \bullet \arrow[r, bend left] \arrow[r, bend right] & \bullet \arrow[r, bend left] \arrow[r, bend right] & \bullet \arrow[r, bend left] \arrow[r, bend right] & \cdots
			\end{tikzcd}
		\end{center}
		Denote by $\G$ the resulting groupoid, and $\ell$ the length function as in \eqref{eq:def-AF-length-function}.
		The length function takes a particularly simple form, namely
		\begin{align*}
			\ell(x,y) = \begin{cases}
				0 & \text{ if $x=y$} \\
				2^m & \text{ if $m = k(x,y)\geq 1$}
			\end{cases}
		\end{align*}
		for $(x,y) \in \G$, where $k(x,y)$ is as in \cref{def:AF-length-function}. The unit space $\Go$ 
		is the Cantor space,
		and we may pick any metric $d$ on $\Go$ 
		metrizing this topology. 
		Choosing the metric stratification as in \eqref{eq:metric-stratification-AF}, we conclude that 
		$(\Lip_c^{\cK}(\G), L^{\cK, n})$
		is a compact quantum metric space for any $n\geq 1$.  Moreover, we obtain compact quantum metric spaces 
		$(\Lip_c^{\cK_m}(\G_m), (L^{\cK, n})_{\vert_{\Lip_c^{\cK_m}(\G_m)}})$
		from the subgroupoids $\G_m$. We see that $\G_m = B_\ell(2^m)$ for $m \in \N$, and therefore by \cref{thm:CQMS-from-AF-groupoids} we have
		\begin{align*}
			\mathrm{dist}_Q (( \Lip_c^{\cK}(\G), L^{\cK, n})&, (\Lip_c^{\cK_m}(\G_m), (L^{\cK, n})_{\vert_{\Lip_c^{\cK_m}(\G_m)}}) \leq \left( \sup_{y \in \Go} \sum_{\substack{(x,y)\in \G_y \\ \ell(x,y) > 2^n}} \ell(x,y)^{-2} \right)^{1/2} \\ 
			&\leq \left( 2\cdot \sum_{k \geq n+1} \frac{1}{2^{2k}} \right)^{1/2} = \frac{2^{1-2n}}{3},
		\end{align*}
		where the second inequality comes from using the fact that for each level there are two edges along which an infinite path might follow. 
	\end{example}

	\printbibliography
	
\end{document}